\documentclass{amsart}
\usepackage{amsmath,amssymb, amsbsy}
\usepackage{hyperref}

\usepackage{color,psfrag}
\usepackage[dvips]{graphicx}
\usepackage{enumerate}
\newcommand{\R}{{\mathbb R}}
\newcommand{\N}{{\mathbb N}}

\newcommand{\weakly}{\rightharpoonup}
\newcommand{\e }{\varepsilon}

\newcommand{\dive }{\mathop{\rm div}}

\renewcommand{\geq }{\geqslant}
\renewcommand{\le }{\leqslant}
\renewcommand{\leq }{\leqslant}

\newtheorem{Theorem}{Theorem}[section]
\newtheorem{Corollary}[Theorem]{Corollary}
\newtheorem{Lemma}[Theorem]{Lemma}

\theoremstyle{definition}

\newtheorem{remark}[Theorem]{Remark}
\begin{document}

\title[Systems with
  Neumann boundary coupling and higher order fractional
  equations]{Unique continuation and
  classification of blow-up profiles for elliptic systems with
  Neumann boundary coupling and applications to higher order fractional
  equations}

\author[Veronica Felli \and Alberto Ferrero]{Veronica Felli \and Alberto Ferrero}
\address{\hbox{\parbox{5.7in}{\medskip\noindent{Veronica Felli, \\
Universit\`a  degli Studi di Milano -- Bicocca,\\
        Dipartimento di Scienza dei Materiali, \\
        Via Cozzi
        55, 20125 Milano, Italy. \\[3pt]
        Alberto Ferrero, \\
        Universit\`a del Piemonte Orientale, \\
        Dipartimento di Scienze e Innovazione Tecnologica, \\
        Viale Teresa Michel 11, 15121 Alessandria, Italy. \\[3pt]
        \em{E-mail addresses: }{\tt veronica.felli@unimib.it,
          alberto.ferrero@uniupo.it}.}}}}

\date{\today}

\thanks{
The authors are partially supported by the INdAM-GNAMPA Project 2018 ``Formula di monotonia e applicazioni: problemi frazionari e stabilit\`a spettrale rispetto a perturbazioni del dominio''.
 V. Felli is partially supported by the PRIN 2015
grant ``Variational methods, with applications to problems in
mathematical physics and geometry''.   
   }

\begin{abstract}
   \noindent
   In this paper we develop a monotonicity formula  for elliptic systems with
  Neumann boundary coupling, proving unique continuation and
  classification of blow-up profiles. As an application, we
  obtain strong unique continuation for some fourth order
   equations and
  higher order
  fractional problems.
\end{abstract}

 \maketitle

\section{Introduction and statement of the main results} \label{s:introduction}

The present paper is devoted to the study of unique continuation from
a boundary point  and
  classification of blow-up profiles for elliptic systems with
  Neumann boundary coupling. Systems of such
a  kind arise from higher order extensions of the fractional
  Laplacian, as first observed in \cite {Y}, where the well known
  Caffarelli-Silvestre extension procedure characterizing the
  fractional Laplacian as the Dirichlet-to-Neumann map in one extra
  spatial dimension was generalized to higher powers of the Laplacian.
  More precisely in \cite {Y} (see also \cite{CY}) it is proved that, if $s\in(1,2)$ and
  $f\in H^s(\R^N)$, then
\begin{equation}\label{eq:44b}
(-\Delta)^s f=K_{s}\lim_{t\to 0^+}t^{b}\frac{\partial(\Delta_b U)}{\partial
  t}
\end{equation}
where $b=3-2s$, $K_{s}$ is a constant depending only on $s$, $\Delta_bU=\Delta U+\frac{b}{t}\frac{\partial
  U}{\partial t}$ and $U$ is the unique solution to the problem
\[
\begin{cases}
  \Delta_b^2 U=0,&\text{in }\R^{N+1}_+=\R^N\times(0,+\infty),\\
U(x,0)=f(x),&\text{in }\R^{N},\\
\lim_{t\to 0^+}t^{b}\frac{\partial U}{\partial
  t}=0 ,&\text{in }\R^{N}.
\end{cases}
\]
Setting $V=\Delta_b U$ and taking into account \eqref{eq:44b}, the above fourth order problem can be rewritten
as the system
\[
 \begin{cases}
  \Delta_b U=V,&\text{in }\R^{N+1}_+,\\
  \Delta_b V=0,&\text{in }\R^{N+1}_+,\\
U(x,0)=f(x),&\text{in }\R^{N},\\
\lim_{t\to 0^+}t^{b}\frac{\partial U}{\partial
  t}=0 ,&\text{in }\R^{N},\\
K_{s} \lim_{t\to 0^+}t^{b}\frac{\partial V}{\partial
  t}=(-\Delta)^s f ,&\text{in }\R^{N}.
\end{cases}
\]
In \cite{Y} an Almgren's frequency formula in the spirit of \cite{almgren} is derived for solutions to
the higher order system
\begin{equation}\label{eq:51}
 \begin{cases}
  \Delta_b U=V,&\text{in }\R^{N+1}_+,\\
  \Delta_b V=0,&\text{in }\R^{N+1}_+,\\
\lim_{t\to 0^+}t^{b}\frac{\partial U}{\partial
  t}=0 ,&\text{in }\R^{N},\\
\lim_{t\to 0^+}t^{b}\frac{\partial V}{\partial
  t}=0 ,&\text{in }\R^{N},
\end{cases}
\end{equation}
obtained by extending $s$-harmonic functions; in the spirit of
Garofalo and Lin \cite{GL}, such monotonicity formula allows proving
a unique  continuation property for solutions to
system \eqref{eq:51}. In
\cite{Y} a strong unique continuation property is also stated for
$s$-harmonic functions.

The main goal of the present paper is to extend, in the case
$s=\frac32$, the monotonicity formula developed in \cite{Y} for the homogeneous case
\eqref{eq:51}  to systems with  a Neumann boundary coupling of the type
\begin{equation}\label{eq:52}
 \begin{cases}
  \Delta U=V,&\text{in }\R^{N+1}_+,\\
  \Delta V=0,&\text{in }\R^{N+1}_+,\\
\frac{\partial U}{\partial
  \nu}=0 ,&\text{in }\R^{N},\\
\frac{\partial V}{\partial
  \nu}=h\, U(\cdot,0),&\text{in }\R^{N},
\end{cases}
\end{equation}
which arise naturally
as extension of fractional equations of the form
\[
(-\Delta)^{3/2}u=a(x)u
\]
once we put $h=-K_{3/2}^{-1}\, a=-2a$. Indeed, by \cite[Proof of Lemma 3.2, Step 6]{FF-prep} we deduce that the constant $C_b$ defined there equals $\sqrt 2$ when $b=0$ and, since it can be shown that $K_s=C_b^{-2}$ with $b=3-2s$, we deduce that $K_{3/2}=\frac 12$.
The frequency function associated to problem \eqref{eq:52} is given by
the ratio of the local
energy over mass near the  fixed point $0\in\R^N$
\begin{equation}\label{eq:59}
\mathcal N(r)=\frac{r^{-N+1} \left[\int_{B_r^+} \left(|\nabla U|^2+|\nabla V|^2+UV\right)dz -\int_{B_r'} h(x)U(x,0)V(x,0)\, dx\right]}{r^{-N}\int_{S_r^+} (U^2+V^2)\, dS},
\end{equation}
where we are denoting
as $z=(x,t)\in \R^N\times\R$ the variable in
$\R^{N+1}=\R^N\times\R$, $dz=dx\,dt$  and, for all $r>0$,
\begin{align*}
&B_r=\{z\in \R^{N+1}:|z|<r\},\quad
B_r^+=\{(x,t)\in B_r:t>0\},\\
&B_r'=\{x\in \R^{N}:|x|<r\}=B_r\cap(\R^N\times\{(x,0):x\in\R^N\}),\\
&S_r^+=\{(x,t)\in \partial B_r:t>0\}.
\end{align*}
The classical approach developed by Garofalo and Lin \cite{GL} to
prove unique continuation through Almgren's monotonicity formula is
based on the validity of doubling type conditions, obtained as a
consequence of boundedness of the quotient $\mathcal N$.
We refer to  \cite{ae,aek,FFFN,FF-edinb,KukavicaNystrom,tz1,tz2}  for
unique continuation from the boundary established via
Almgren monotonicity formula.

 While in the local case doubling conditions are enough to
establish unique continuation, in the fractional case they provide
unique continuation only for the extended local problem and not for
the fractional one. Such difficulty was overcome in \cite{FF} for the
fractional Laplacian $(-\Delta)^{s}$ with $s\in(0,1)$,  by a
fine blow-up analysis and a precise classification of the possible
blow-up limit profiles in terms of a Neumann eigenvalue problem on the
half-sphere.

The problem of unique continuation for fractional
laplacians with power $s\in(0,1)$ was also studied in \cite{Ruland} in presence of rough
potentials using Carleman estimates and in
\cite{Yu} for fractional operators with variable coefficients using
an Almgren type monotonicity formula.
As far as higher fractional powers of the laplacian, the main
contribution to the problem of unique continuation is due to Seo in
papers \cite{seo0, seo1, seo2}, through Carleman inequalities; in particular
papers \cite{seo0, seo1, seo2} consider fractional Schr\"odinger operators with
potentials in Morrey spaces and prove a \emph{weak} unique
continuation result, i.e. vanishing of solutions which are zero on an
open set; we recall that the \emph{strong} unique
continuation  property instead requires the weaker assumption of
infinite vanishing order at some point.

We observe that the presence of a coupling Neumann term in system
\eqref{eq:52} produces substancial additional difficulties with respect to
the extension problem corresponding to the  lower order
fractional case $s\in(0,1)$ and consisting in a single equation
associated with a Neumann boundary condition. In particular the proof of  a
monotonicity formula for \eqref{eq:52}  is made quite delicate by the
appearance in the derivative of the frequency $\mathcal N$ of a term
of the type
\[
-r\int_{\partial B_r'}huv\,dS'+2\int_{B_r'}hu\,x\cdot
  \nabla_x v\,dx,
\]
see Lemma \ref{mono}. While in the lower order case we have only one
component $u=v$ so that an integration by parts allows rewriting the
above sum as an integral over $B_r'$, in the case of two components
$u,v$ this is no more possible and an estimate of the
integral over ``the boundary of the boundary''
$\int_{\partial B_r'}huv\,dS'$ is required. The method developed
here to overcome this difficulty is based on estimates in terms of
boundary integrals (see Lemma \ref{l:STIMA_DI_nu_2}) and represents
one of the main technical novelty of the present paper in the context
of monotonicity formulas; we think that
this procedure could have future applications in the extension of some
of the results of \cite{FF} to rough potentials, since it could avoid
the integration by parts needed to write  the
above sum as an integral over $B_r'$, which requires differentiability
of the potential $h$.

Let $N>3$,
$R>0$, and $(U,V)\in H^1(B_R^+)\times H^1(B_R^+)$ be  a weak solution to the system
\begin{equation}\label{eq:system}
  \begin{cases}
\Delta U=V,&\text{in }B_R^+,\\
\Delta V=0,&\text{in }B_R^+,\\
\frac{\partial U}{\partial \nu}= 0, &\text{in } B_R',\\
\frac{\partial V}{\partial \nu}=hu, &\text{in } B_R',
  \end{cases}
\end{equation}
where $u=U(\cdot,0)$ (trace of $U$ on $B_R'$) and $h\in
C^1(B_R')$. We also denote  $v=V(\cdot,0)$ (trace of $V$ on
$B_R'$).  By a weak solution to the system \eqref{eq:system} we mean a
couple $(U,V)\in H^1(B_R^+)\times H^1(B_R^+)$ such that, for every
$\varphi\in  H^1(B_R^+)$ having zero trace on $S_R^+$,
\[
\begin{cases}
\int_{B_R^+}\nabla U(z)\cdot\nabla
\varphi(z)\,dz=-\int_{B_R^+}V(z)\varphi(z)\,dz,\\
\int_{B_R^+}\nabla V(z)\cdot\nabla
\varphi(z)\,dz=\int_{B_R'}h(x)u(x)\mathop{\rm Tr}\varphi (x)\,dx,
\end{cases}
\]
where $\mathop{\rm Tr}\varphi$ is the trace of $\varphi$ on $B_R'$.

Our first result is an asymptotic expansion of nontrivial solutions to
\eqref{eq:system}; more precisely we prove that blow-up profiles can
be described as combinations of spherical harmonics symmetric with
respect to the equator $t=0$.

Let $-\Delta _{\mathbb{S}^{N}}$ denote the Laplace Beltrami operator
on the $N$-dimensional unit sphere $\mathbb{S}^{N}$. It is well known
that the  eigenvalues
of $-\Delta _{\mathbb{S}^{N}}$  are given by
\begin{equation*}
\lambda _{\ell }=(N-1+\ell )\ell ,\quad \ell =0,1,2,\dots .
\end{equation*}
For every $\ell\in\N$, it is easy to verify that there exists a
spherical harmonic on $\mathbb{S}^{N}$ of degree $\ell$ which is
symmetric with respect to the equator $t=0$\footnote{It is enough to
  take a homogeneous harmonic polynomial $P=P(x_1,x_2,\dots,x_N)$ in
  $N$ variables of degree $\ell$ and consider the  homogeneous harmonic
  polynomial in $N+1$ variables
  $P'(x_1,x_2,\dots,x_N,x_{N+1})=P(x_1,x_2,\dots,x_N)$, whose
  restriction to $\mathbb{S}^{N}$ satisfies the required properties.}. Therefore
the eigenvalues of the  problem
\begin{align}\label{eq:sph_eig}
  \begin{cases}
    -\Delta_{{\mathbb S}^{N}}\psi=\lambda\,\psi, &\text{in }{\mathbb S}^{N}_+,\\[5pt]
\nabla_{{\mathbb
    S}^{N}}\psi\cdot {\mathbf
  e}=0,&\text{on }\partial {\mathbb S}^{N}_+,
  \end{cases}
\end{align}
with
\begin{equation*}
  {\mathbb S}^{N}_+=\{(\theta_1,\theta_2,\dots, \theta_{N+1})\in
  {\mathbb S}^{N}:\theta_{N+1}>0\},\quad {\mathbf e}=(0,0,\dots,0,1),
\end{equation*}
are given by the sequence $\{\lambda_\ell:\ell=0,1,2,\dots\}$; for
every $\ell$, $\lambda_\ell$ has finite multiplicity $M_\ell$ as an
eigenvalue of \eqref{eq:sph_eig}. For every $\ell \geq 0$, let
$\{Y_{\ell ,m}\}_{m=1,2,\dots ,M_{\ell }}$ be a
$L^{2}(\mathbb{S}^{N}_+)$-orthonormal basis of the eigenspace of
\eqref{eq:sph_eig} associated to $\lambda _{\ell }$ with $Y_{\ell ,m}$
being spherical harmonics of degree $\ell $.

We note that, if $\Psi$ is an eigenfunction of \eqref{eq:sph_eig}, then
$\Psi\not\equiv0$ on $\partial {\mathbb
  S}^{N}_+={\mathbb S}^{N-1}$; indeed, by
unique continuation, $\Psi$ and $\nabla_{{\mathbb
    S}^{N}}\Psi\cdot {\mathbf  e}$ can not both vanish on $\partial {\mathbb
  S}^{N}_+$. In particular
$Y_{\ell,m}\not\equiv0$ on $\partial {\mathbb
  S}^{N}_+={\mathbb S}^{N-1}$ for all $\ell\in\N$ and $1\leq m\leq M_\ell$.

\begin{Theorem} \label{t:asym}
 Let $(U,V)\in H^1(B_R^+)\times H^1(B_R^+)$ be  a weak solution to
\eqref{eq:system} such that $(U,V)\not=(0,0)$. Then there exists
$\ell\in\N$ such that
\[
\lambda^{-\ell}U(\lambda z)\to \widehat U(z),\quad
\lambda^{-\ell}V(\lambda z)\to \widehat V(z),
\]
as $\lambda\to0^+$ strongly in $H^1(B_1^+)$, where
\[
\widehat U(z)=|z|^{\ell}
\sum_{m=1}^{M_\ell}\alpha_{\ell,m}Y_{\ell,m}\Big(\frac
z{|z|}\Big),\quad
 \widehat V(z)=|z|^{\ell}
\sum_{m=1}^{M_\ell}\alpha'_{\ell,m}Y_{\ell,m}\Big(\frac
z{|z|}\Big),
\]
\begin{align}
 \label{eq:45} \alpha_{\ell,m}&=R^{-\ell}\int_{{\mathbb S}^{N}_+}U(R\,\theta)
  Y_{\ell,m}(\theta)\,dS-
\frac{R^{-N-2\ell+1}}{N+2\ell-1}\int_0^R t^{N+\ell}\left(\int_{{\mathbb S}^{N}_+}V(t\,\theta)
    Y_{\ell,m}(\theta)\,dS\right)\,dt\\
\notag&\qquad\qquad
+\int_0^R\frac{t^{-\ell+1}}{2\ell+N-1}\left(\int_{{\mathbb S}^{N}_+}V(t\,\theta)
    Y_{\ell,m}(\theta)\,dS\right)\,dt,\\
 \label{eq:46} \alpha'_{\ell,m}&=R^{-\ell}\int_{{\mathbb S}^{N}_+}V(R\,\theta)
  Y_{\ell,m}(\theta)\,dS\\
\notag&\qquad\qquad-
\frac{R^{-N-2\ell+1}}{N+2\ell-1}\int_0^R t^{N+\ell-1}\left(
\int_{{\mathbb
      S}^{N-1}}h(t\theta')U(t\theta',0) Y_{\ell,m}(\theta',0) \,dS'\right)\,dt\\
\notag&\qquad\qquad
+\int_0^R\frac{t^{-\ell}}{2\ell+N-1}\left(
\int_{{\mathbb
      S}^{N-1}}h(t\theta')U(t\theta',0)Y_{\ell,m}(\theta',0)\,dS'
\right)\,dt,
\end{align}
and
\[
\sum_{m=1}^{M_\ell}((\alpha_{\ell,m})^2+(\alpha'_{\ell,m})^2)\neq0.
\]
\end{Theorem}

A first remarkable consequence of Theorem \ref{t:asym} is the validity
of a {\em{strong unique continuation
  property}} (from the boundary point $0$) for solutions to \eqref{eq:system}.

\begin{Theorem}\label{t:sun-ext}
 Let $(U,V)\in H^1(B_R^+)\times H^1(B_R^+)$ be  a weak solution to
\eqref{eq:system}. If
\begin{equation}\label{eq:case2}
U(z)=o(|z|^n)\quad\text{as }|z|\to 0\text{ for all $n\in
  \N$}
\end{equation}
then $U\equiv V\equiv 0$ in $B_R^+$.
\end{Theorem}
We observe that in the case of a single equation a blow result
as the one stated in Theorem \ref{t:asym} directly yields the
strong unique continuation: indeed, if the solution has a precise
vanishing order it cannot vanish of any order. On the other hand, in
the case of a system of type \eqref{eq:system}, the blow-up Theorem
\ref{t:asym} ensures that the couple of the limit profiles
$(\widehat U, \widehat V)$ is not
trivial, i.e. at least one of the two components
$U,V$ has  a precise
vanishing order; hence some further analysis is needed to deduce
strong unique continuation from Theorem
\ref{t:asym}.

System \eqref{eq:system} is related to fourth order elliptic equations
arising in Caffarelli-Silvestre type extensions for higher order
fractional laplacians in the spirit of \cite{Y}. Let us define $\mathcal D$ as the completion of
\begin{equation} \label{eq:space}
\mathcal T:=\left\{U\in C^\infty_c(\overline{\R^{N+1}_+}):U_t\equiv 0 \ \text{on} \
\R^N\times\{0\}
\right\}
\end{equation}
with respect to the norm
\[
\|U\|_{\mathcal D}=\bigg(\int_{\R^{N+1}_+} |\Delta
U(x,t)|^2\,dx\,dt\bigg)^{\!\!1/2}.
\]
By \cite{FF-prep} there exists a well defined continuous  trace
map
\[
\mathop{\rm Tr}:\mathcal D\to \mathcal D^{3/2,2}(\R^N),
\]
where  the space $\mathcal D^{3/2,2}(\R^N)$ is defined as
the completion of $C^\infty_c(\R^N)$ with respect to the scalar product
\begin{equation} \label{eq:scalar-product}
(u,v)_{\mathcal D^{3/2,2}(\R^N)}:=\int_{\R^N} |\xi|^{3}\, \widehat
u(\xi)\overline{\widehat v(\xi)} \, d\xi.
\end{equation}
In \eqref{eq:scalar-product} $\widehat u$ denotes the Fourier transform of $u$ in $\R^N$:
$$
\widehat u(\xi)=\frac 1{(2\pi)^{N/2}} \int_{\R^N} e^{-i\xi x}u(x)\, dx \, .
$$
Moreover in \eqref{eq:scalar-product} we denoted by
$\overline{\widehat v(\xi)}$ the complex conjugate of
  $\widehat v(\xi)$.

We observe that, since $u$ and $v$ are real functions, \eqref{eq:scalar-product} is really a scalar product although their respective Fourier transforms are complex functions.

As a corollary of Theorem \ref{t:asym} we derive sharp asymptotic
estimates and  a strong unique continuation
principle for weak $\mathcal D$-solutions to the fourth order elliptic
problem
\begin{equation}\label{eq:44}
  \begin{cases}
    \Delta^2 U=0,&\text{in }\R^{N+1}_+,\\[3pt]
\frac{\partial U}{\partial\nu}=0,&\text{in }\R^N \, , \\[3pt]
\frac{\partial(\Delta U)}{\partial\nu}=h\, \mathop{\rm
  Tr}(U),&\text{in }\Omega \, .
  \end{cases}
\end{equation}
By a weak $\mathcal D$-solution to \eqref{eq:44} we mean some $U\in
\mathcal D$ such that
\[
\int_{\R^{N+1}_+}\Delta U(x,t) \Delta \varphi(x,t)\,dx\,dt=-\int_\Omega
h(x) \mathop{\rm Tr}U(x)
\mathop{\rm Tr}\varphi(x)\,dx
\]
for all $\varphi\in\mathcal D$ such that $\mathop{\rm supp}(\mathop{\rm Tr}\varphi)\subset\Omega$.
\begin{Theorem}\label{t:asym-ho}

\begin{enumerate}[\rm (i)]
\item  Let $U\in\mathcal D$, $U\not\equiv0$, be a nontrivial weak solution
  to \eqref{eq:44} for some $h\in C^1(\Omega)$, with
  $\Omega$ being an open bounded set in  $\R^N$ such that $0\in\Omega$.
 Then there exists
$\ell\in\N$ such that
\[
\lambda^{-\ell}U(\lambda z)\to
|z|^{\ell}
\sum_{m=1}^{M_\ell}\alpha_{\ell,m}Y_{\ell,m}\Big(\frac
z{|z|}\Big), \quad
\lambda^{-\ell}\Delta U(\lambda z)\to
|z|^{\ell}
\sum_{m=1}^{M_\ell}\alpha'_{\ell,m}Y_{\ell,m}\Big(\frac
z{|z|}\Big),
\]
strongly in $H^1(B_1^+)$, where
$\sum_{m=1}^{M_\ell}((\alpha_{\ell,m})^2+(\alpha'_{\ell,m})^2)\neq0$
and $\alpha_{\ell,m},\alpha'_{\ell,m}$ are given in
\eqref{eq:45}--\eqref{eq:46} with $V=\Delta U$.
\item  If $U\in \mathcal D$ is  a weak solution to \eqref{eq:44} such
  that
\begin{equation*}
U(z)=o(|z|^n)\quad\text{as }|z|\to 0\text{ for all $n\in
  \N$},
\end{equation*}
then $U\equiv 0$ in $B_R^+$.
\end{enumerate}
\end{Theorem}

As mentioned above, a motivation for the study of higher order equations of type
\eqref{eq:44} and consequently of systems \eqref{eq:system} comes from
the interest in higher order fractional laplacians and their
characterization as a Dirichlet-to-Neumann map in the spirit of
\cite{CS}.

Let us consider
 the fractional laplacian $(-\Delta)^{3/2}$ defined as
\[
\widehat{(-\Delta)^{\frac32} u}(\xi)=|\xi|^{3}\widehat u(\xi) \, .
\]
We also consider the space
 $\mathcal D^{1/2,2}(\R^N)$ given by
the completion of $C^\infty_c(\R^N)$ with respect to the scalar product
\begin{equation*}
(u,v)_{\mathcal D^{1/2,2}(\R^N)}:=\int_{\R^N} |\xi|\, \widehat
u(\xi)\overline{\widehat v(\xi)} \, d\xi.
\end{equation*}

\begin{Theorem}\label{t:frac}
  For $N>3$, let $\Omega\subseteq\R^N$ be open, $a\in C^1(\Omega)$,
  and $u\in \mathcal D^{3/2,2}(\R^N)$ be a weak solution to the
  problem
\begin{equation}\label{eq:48}
(-\Delta)^{3/2}u=a\,u,\quad\text{in }\Omega,
\end{equation}
i.e.
\[
(u,\varphi)_{\mathcal
  D^{3/2,2}(\R^N)}=\int_{\Omega}a\,u\,\varphi\,dx\quad\text{for all }
\varphi\in C^\infty_c(\Omega).
\]
Let us also assume that
\begin{equation}\label{eq:53}
  (-\Delta)^{3/2}u\in (\mathcal D^{1/2,2}(\R^N))^\star,
\end{equation}
where $(\mathcal D^{1/2,2}(\R^N))^\star$ denotes the dual space of $\mathcal D^{1/2,2}(\R^N)$,  in the sense that the linear functional
  $\varphi\mapsto \int_{\R^N} |\xi|^{3} \widehat
  u(\xi)\overline{\widehat \varphi(\xi)} \, d\xi$,
  $\varphi\in C^\infty_c(\R^N)$, is continuous with respect to the
  norm induced by $\mathcal D^{1/2,2}(\R^N)$.
\begin{enumerate}[\rm (i)]
\item If $u$ vanishes at some point $x_0\in\Omega$ of infinite
order, i.e. if
\begin{equation}\label{eq:47}
u(x)=o(|x-x_0|^n)\quad\text{as }x\to x_0\text{ for every }n\in\N,
\end{equation}
then $u\equiv 0$ in $\Omega$.
\item If $u$ vanishes on a set $E\subset\Omega$ of positive Lebesgue
  measure, then $u\equiv 0$ in $\Omega$.
\end{enumerate}
\end{Theorem}

\begin{remark} We observe that assumption \eqref{eq:53} is satisfied
  in each of the following cases:
\begin{enumerate}[(i)]
\item $u\in \mathcal D^{5/2,2}(\R^N)$;
\item $u\in \mathcal D^{3/2,2}(\R^N)$ solves \eqref{eq:48} with
$\Omega=\R^N$ and $a\in L^{N/2}(\R^N)\cap C^1(\R^N)$.
\end{enumerate}
\end{remark}

The proof of Theorem \ref{t:frac} is based on Theorem~\ref{t:asym-ho}
and the generalization of the Caffarelli-Silvestre extension to higher
order fractional laplacians given in \cite{Y}, see also
\cite{FF-prep}. Indeed, according to \cite{Y}, we have that if $u$
solves \eqref{eq:48}, then $u$ is the trace on $\R^N\times\{0\}$ of
some $U\in\mathcal D$ solving \eqref{eq:44}   with $h=-2a$.

We observe that the  unique continuation result stated in Theorem
\ref{t:frac} does not overlap with the results in \cite{seo0, seo1,
  seo2}. Indeed,  from one hand \cite{seo0, seo1,
  seo2} consider more general potentials; on the other hand we obtain
here
a \emph{strong} unique continuation and a unique continuation from
sets of positive measure, which are stronger results than the weak unique
continuation obtained in   \cite{seo0, seo1,
  seo2}. We also observe that we assume  that equation \eqref{eq:48}
is satisfied only on the set $\Omega$ and not in the whole $\R^N$.

The paper is organized as follows. In section \ref{sec:monot-argum} we
develop the monotonicity argument, proving in particular the existence
of a finite limit for the frequency function \eqref{eq:59} as
$r\to0^+$. In section \ref{sec:blow-up-analysis} we carry out a careful
blow-up analysis for scaled solutions, which allows proving Theorem
\ref{t:asym} and, as a consequence, Theorem \ref{t:sun-ext}. Finally
section \ref{sec:appl-fourth-order} is devoted to
applications of Theorem \ref{t:asym} to fourth order problems
\eqref{eq:44}  and higher order
  fractional problems \eqref{eq:48}, with the proofs of Theorems
  \ref{t:asym-ho} and \ref{t:frac}.

\section{The monotonicity argument}\label{sec:monot-argum}
For all $r\in(0,R)$ we define the functions
\begin{equation} \label{eq:D(r)}
D(r)=r^{-N+1} \left[\int_{B_r^+} \left(|\nabla U|^2+|\nabla
    V|^2+UV\right)dz
-\int_{B_r'} h(x)u(x)v(x)\, dx  \right]
\end{equation}
and
\begin{equation} \label{eq:H(r)}
H(r)=r^{-N}\int_{S_r^+} (U^2+V^2)\, dS.
\end{equation}
We define the space $\mathcal
D^{1,2}(\R^{N+1}_+)$ as the completion of the space
$C^\infty_c(\overline{\R^{N+1}_+})$ with respect to the norm
$$
\|U\|_{\mathcal D^{1,2}(\R^{N+1}_+)}:=\left(\int_{\R^{N+1}_+}
|\nabla U|^2 dz\right)^{1/2} \qquad \text{for any } U\in
C^\infty_c(\overline{\R^{N+1}_+}) \, .
$$
From \cite{BCdPS}, we have that there exists a constant
$K>0$ such that
\begin{equation} \label{eq:C-b}
K \|{\rm Tr\, }U\|_{\mathcal D^{\frac12,2}(\R^N)}\le \|U\|_{\mathcal
D^{1,2}(\R^{N+1}_+)} \qquad \text{for any } U\in \mathcal
D^{1,2}(\R^{N+1}_+)  .
\end{equation}
Here we are denoting as ${\rm Tr\, }$ the trace operator
${\rm Tr\, }:\mathcal D^{1,2}(\R^{N+1}_+)\to \mathcal D^{\frac12,2}(\R^N)$. We recall that, for all
$\gamma<\frac N2$ the following Sobolev embedding holds: there exists
a positive constant $S(N,\gamma)$ depending only on $N$ and $\gamma$
such that
\begin{equation} \label{eq:Sobolev-frazionario}
S(N,\gamma) \|u\|_{L^{2^*(N,\gamma)}(\R^N)}^2\le \|u\|_{\mathcal
D^{\gamma,2}(\R^{N})}^2 \qquad \text{for any } u\in
C^\infty_c(\R^N)
\end{equation}
where $2^*(N,\gamma)=2N/(N-2\gamma)$.
Moreover the following Hardy type inequality due to Herbst
\cite{herbst} holds: there exists $\Lambda>0$
\begin{equation}\label{eq:frac_hardy}
  \Lambda\int_{\R^N}\frac{\varphi^2(x)}{|x|}\,dx\leq\|\varphi\|_{\mathcal D^{1/2,2}(\R^N)}^2
  ,\quad\text{for all }\varphi\in \mathcal D^{1/2,2}(\R^N).
\end{equation}
Combining \eqref{eq:C-b} and \eqref{eq:Sobolev-frazionario} we obtain
that
\begin{equation} \label{eq:Sobolev-frazionario-bis}
S\big(N,\tfrac12\big)K^2 \, \|{\rm Tr\, }U\|_{L^{\frac{2N}{N-1}}(\R^N)}^2\le
\|U\|_{\mathcal D^{1,2}(\R^{N+1}_+)}^2 \qquad \text{for any }
U\in \mathcal D^{1,2}(\R^{N+1}_+) .
\end{equation}
Similarly, combining \eqref{eq:C-b} with \eqref{eq:frac_hardy}, we
infer
\begin{equation} \label{eq:Hardy-frazionario}
\Lambda K^2 \int_{\R^N} \frac{|{\rm
Tr\,}U|^2}{|x|} \, dx \le \|U\|_{\mathcal D^{1,2}(\R^{N+1}_+)}^2  \qquad \text{for any } U\in \mathcal
D^{1,2}(\R^{N+1}_+) .
\end{equation}
We recall the following lemmas from \cite{FF}, which provide Sobolev
and Hardy type trace inequalities with
boundary terms in $N+1$-dimensional half-balls.
\begin{Lemma}[\cite{FF} Lemma 2.6] \label{l:sobolev} For any $r>0$ and any $U\in H^1(B_r^+)$ we have
\begin{equation*}
\widetilde S \left(\int_{B_r'}
|u|^{\frac{2N}{N-1}}dx\right)^{\!\!\frac{N-1}{N}} \le \int_{B_r^+}
 |\nabla U|^2 dz+\frac{N-1}{2r} \int_{S_r^+}  U^2 dS
\end{equation*}
where $u={\rm Tr\, }U$ and $\widetilde S$ is a positive
constant depending only on $N$.
\end{Lemma}

\begin{Lemma}[\cite{FF} Lemma 2.5]\label{l:hardyball} For any $r>0$ and any $U\in
  H^1(B_r^+)$ we have that
\begin{equation*}
\widetilde\Lambda \int_{B_r'} \frac{u^2}{|x|} \,
dx\le \int_{B_r^+} |\nabla U|^2 dz+\frac{N-1}{2r}
\int_{S_r^+} U^2 dS
\end{equation*}
where $u={\rm Tr\,} U$ and $\widetilde \Lambda$ is a
positive constant depending only on $N$.
\end{Lemma}

The following Poincar\'e type
inequality on half-balls will be useful in the sequel.

\begin{Lemma}\label{l:Poincare}
  For every $r>0$ and $W\in H^1(B_r^+)$ we have that
\[
\frac{N}{r^2}\int_{B_r^+}W^2(z)\,dz\leq \frac{1}{r}\int_{S_r^+}W^2(z)\,dS+\int_{B_r^+}|\nabla
W(z)|^2\,dz.
\]
\end{Lemma}
\begin{proof}
  From the Divergence Theorem we have that
\begin{align*}
(N+1)\int_{B_r^+}W^2(z)\,dz&=
\int_{B_r^+}\left(\dive(W^2z)-2W\nabla W\cdot z\right)\,dz\\
&=r\int_{S_r^+}W^2(z)\,dS-2 \int_{B_r^+}W\nabla W\cdot z\,dz\\
&\leq r\int_{S_r^+}W^2(z)\,dS+\int_{B_r^+}W^2(z)\,dz+r^2 \int_{B_r^+}|\nabla W|^2\,dz
\end{align*}
thus yielding the stated inequality.
\end{proof}

The following lemma contains a Pohozaev type identity for solutions to
system \eqref{eq:system}.
\begin{Lemma} Let $(U,V)\in H^1(B_R^+)\times H^1(B_R^+)$ be  a weak solution to
\eqref{eq:system}.
Then for a.e. $r\in(0,R)$
\begin{equation} \label{Pohozaev-1}
\int_{B_r^+} \left(|\nabla U|^2+|\nabla
V|^2+UV\right)dz=\int_{S_r^+} \left(\frac{\partial U}{\partial
\nu} U+\frac{\partial V}{\partial \nu}V\right)dS  +\int_{B_r'}
h(x)u(x)v(x)\, dx
\end{equation}
and
\begin{align} \label{Pohozaev-2}
-\frac{N-1}2 & \int_{B_r^+}  \left(|\nabla U|^2+|\nabla
V|^2\right)dz+\int_{B_r^+}  V(z\cdot \nabla U)\, dz +\frac r2
\int_{S_r^+}  \left(|\nabla U|^2+|\nabla V|^2\right)dS \\
\notag & =\int_{B_r'} h(x)u(x)(x\cdot \nabla_x v)\,
dx+ r\int_{S_r^+}  \left(\left|\frac{\partial
U}{\partial\nu}\right|^2+\left|\frac{\partial
V}{\partial\nu}\right|^2\right)dS
\end{align}
where $u(x):=U(x,0)$ and $v(x)=V(x,0)$.
\end{Lemma}
\begin{proof}
  Identity \eqref{Pohozaev-1} follows by testing the equation for $U$ with $U$
  and the equation for $V$ with $V$ and by integrating by parts over
  $B_r^+$.

To prove \eqref{Pohozaev-2} we first observe that $U,V\in
H^2(B_r^+)$ for all $r\in(0,R)$. Indeed, since $\frac{\partial U}{\partial \nu}= 0$ on $B_R'$, the
function
\[
\widetilde U(x,t)=
\begin{cases}
  U(x,t),&\text{if }t>0,\\
  U(x,-t),&\text{if }t<0,
\end{cases}
\]
satisfies the equation $\Delta \widetilde U=\widetilde V$, where $\widetilde V(x,t)=
  V(x,t)$ if $t>0$ and $\widetilde V(x,t)=V(x,-t)$ if $t<0$. Since
  $\widetilde V\in L^2(B_R)$, by classical elliptic regularity
   we have that $\widetilde U\in H^2(B_r)$  and  hence $U\in
H^2(B_r^+)$ for all $r\in(0,R)$. By the Gagliardo Trace Theorem we
have that $u=\mathop{\rm Tr}U\in H^{1/2}(B_r')$ for all $r\in(0,R)$ . Since $h\in
C^1(B_R')$ we have that $hu \in  H^{1/2}(B_r')$ for all
$r\in(0,R)$. Therefore, for all $r\in(0,R)$, $V$ satisfies
\[
\begin{cases}
  \Delta V=0,&\text{in }B_r^+,\\
\frac{\partial V}{\partial \nu}\in H^{1/2}(B_r').&
\end{cases}
\]
From elliptic regularity under Neumann boundary conditions (see in
particular \cite[Theorem 8.13]{Salsa}) we conclude that $V\in H^2(B_r^+)$ for all $r\in(0,R)$.

Since, for every $r\in(0,R)$, $U,V\in
H^2(B_r^+)$, we can test  the equation for $U$ with $\nabla U\cdot z$
(which belongs to $H^1(B_r^+)$)
  and the equation for $V$ with $\nabla V\cdot z$
(which belongs to $H^1(B_r^+)$), thus obtaining \eqref{Pohozaev-2}.
\end{proof}

\begin{Lemma}\label{l:Hpos}
 Let $(U,V)\in H^1(B_R^+)\times H^1(B_R^+)$ be  a weak solution to
\eqref{eq:system} such that $(U,V)\not=(0,0)$ (i.e. $U$ and $V$ are not
both identically null).
Let $D=D(r)$ and $H=H(r)$ be the functions
defined in \eqref{eq:D(r)} and \eqref{eq:H(r)}. Then there exists
$r_0\in(0,R)$ such that $H(r)>0$ for any $r\in (0,r_0)$.
\end{Lemma}

\begin{proof} Suppose by contradiction that for any $r_0>0$ there
exists $r\in (0,r_0)$ such that $H(r)=0$. Then there exists a sequence
$r_n\to 0^+$ such that  $H(r_n)=0$, i.e. $U=V=0$ on $S_{r_n}^+$. From
\eqref{Pohozaev-1} it follows that
\begin{equation}\label{eq:2}
\int_{B_{r_n}^+} \left(|\nabla U|^2+|\nabla
V|^2+UV\right)dz=\int_{B_{r_n'}}
h(x)u(x)v(x)\, dx.
\end{equation}
 From \eqref{eq:2}, Lemma \ref{l:Poincare}, and Lemma
\ref{l:hardyball} it follows that
\begin{align*}
\left(1-\frac{r_n^2}{2N}\right)& \int_{B_{r_n}^+} \left(|\nabla U|^2+|\nabla
V|^2\right)dz\leq \int_{B_{r_n}^+} \left(|\nabla U|^2+|\nabla
V|^2+UV\right)dz\\
&=\int_{B_{r_n'}}h(x)u(x)v(x)\, dx\leq {\rm
  const\,}r_n\left(\int_{B_{r_n}'}\frac{u^2}{|x|}\,dx+
\int_{B_{r_n}'}\frac{v^2}{|x|}\,dx\right)
\\
&\leq {\rm
  const\,}r_n
\int_{B_{r_n}^+} \left(|\nabla U|^2+|\nabla
V|^2\right)dz.
\end{align*}
Since $r_n\to0^+$ as $n\to+\infty$, the above inequality implies that $\int_{B_{r_n}^+} \left(|\nabla U|^2+|\nabla
V|^2\right)dz=0$ for $n$ sufficiently large. Hence, in view of
Lemma \ref{l:Poincare}, $U\equiv V\equiv 0$ in $B_{r_n}^+$. Classical unique
continuation principles then imply that $U\equiv V\equiv 0$ in $B_R^+$
giving rise to a contradiction.
\end{proof}

\begin{Lemma}\label{l:Dgeq}
  Letting $(U,V)\in H^1(B_R^+)\times H^1(B_R^+)$ be as in Lemma
  \ref{l:Hpos} and  $D,H$ as in \eqref{eq:D(r)}--\eqref{eq:H(r)},
  there holds
\begin{align}\label{eq:5}
&D(r)\geq r^{1-N} \left(\int_{B_{r}^+} \left(|\nabla U|^2+|\nabla
V|^2\right)dz\right)(1+O(r))-H(r)O(r),\\
\label{eq:6}
&D(r)\geq Nr^{-1-N} \left(\int_{B_{r}^+} (U^2+V^2)dz\right)(1+O(r))-H(r)O(1),
\end{align}
as $r\to0^+$.
\end{Lemma}
\begin{proof}
  From Lemma \ref{l:Poincare} we have that
\begin{equation}\label{eq:4}
\int_{B_r^+}(U^2+V^2)\,dz\leq \frac{r^{1+N}}{N}H(r)+\frac{r^2}{N}\int_{B_{r}^+} \left(|\nabla U|^2+|\nabla
V|^2\right)dz.
\end{equation}
From \eqref{eq:4} it follows that
\begin{equation}\label{eq:1}
\left|\int_{B_r^+}UV\,dz\right|\leq \frac{r^{1+N}}{2N}H(r)+\frac{r^2}{2N}\int_{B_{r}^+} \left(|\nabla U|^2+|\nabla
V|^2\right)dz
\end{equation}
whereas Lemma \ref{l:hardyball} implies that, for all $r\in(0,r_0)$,
\begin{align}\label{eq:3}
\left|\int_{B_r'}huv\,dx\right|&\leq \|h\|_{L^{\infty}(B_{r_0}')}\frac{r}2\int_{B_r'}
 \frac{u^2+v^2}{|x|} \,dx\\
\notag&\leq  \|h\|_{L^{\infty}(B_{r_0}')}\frac{r}{2\widetilde\Lambda}\left(\int_{B_{r}^+} \left(|\nabla U|^2+|\nabla
V|^2\right)dz\right)
+ \|h\|_{L^{\infty}(B_{r_0}')}\frac{N-1}{4\widetilde\Lambda}r^NH(r).
\end{align}
From \eqref{eq:1} and \eqref{eq:3} it follows that
\begin{align*}
  D(r)&\geq r^{1-N} \left(\int_{B_r^+}(|\nabla U|^2+|\nabla
V|^2)dz\right)\left(1-\frac{r^2}{2N}-
        \|h\|_{L^{\infty}(B_{r_0}')}\frac{r}{2\widetilde\Lambda}\right)\\
&\qquad -rH(r)\left(\frac{r}{2N}+\|h\|_{L^{\infty}(B_{r_0}')}\frac{N-1}{4\widetilde\Lambda}\right).
\end{align*}
The proof of \eqref{eq:5} is thereby complete. Estimate \eqref{eq:6}
follows by combination of \eqref{eq:5} and \eqref{eq:4}.
\end{proof}
\begin{remark}\label{rem:st}
  We observe that estimates \eqref{eq:5} and \eqref{eq:6} can be
  rewritten as
\begin{align}\label{eq:8}
&\int_{B_{r}^+} \left(|\nabla U|^2+|\nabla
V|^2\right)dz\leq D(r) r^{N-1}(1+O(r))+H(r)O(r^N),\\
\label{eq:9}
&\int_{B_{r}^+} (U^2+V^2)dz\leq \frac1N r^{N+1}D(r) (1+O(r))+H(r)O(r^{N+1}),
\end{align}
as $r\to0^+$.
\end{remark}

\begin{Lemma} \label{l:hprime}
We have that   $H\in
  W^{1,1}_{{\rm loc}}(0,R)$ and
\begin{align}\label{H'}
 & H'(r)=2r^{-N} \int_{S_r^+}\big(U\tfrac{\partial
    U}{\partial\nu}+V\tfrac{\partial
    V}{\partial\nu}\big)
  \, dS, \quad \text{in a distributional sense and for a.e. $r\in (0,R)$},\\
 &\label{H'2} H'(r)=\frac2r D(r), \quad \text{for every } r\in (0, R).
\end{align}
\end{Lemma}

\begin{proof} See the proof of \cite[Lemma 3.8]{FF}.
\end{proof}

\begin{Lemma}\label{l:dprime}
 The function $D$ defined in \eqref{eq:D(r)} belongs to $W^{1,1}_{{\rm\
      loc}}(0, R)$ and
\begin{align}\label{D'F}
  D'(r)=&\frac{2}{r^{N-1}} \int_{S_r^+}
          \left(\left|\frac{\partial U}{\partial \nu}\right|^2+
          \left|\frac{\partial V}{\partial \nu}\right|^2\right) dS
+\frac{1}{r^{N-1}} \int_{S_r^+}  UV\,dS\\
\notag&-\frac{2}{r^{N}}\int_{B_r^+}V\,\nabla U\cdot
  z\,dz-\frac{N-1}{r^{N}}\int_{B_r^+} UV\,dz\\
\notag&+\frac{N-1}{r^{N}}\int_{B_r'}huv-\frac{1}{r^{N-1}}\int_{\partial
  B_r'}huv\,dS'+2\frac1{r^{N}}\int_{B_r'}hu(x\cdot\nabla_x v)\,dx
\end{align}
in a distributional sense and for a.e. $r\in (0,R)$.
\end{Lemma}
\begin{proof}
  For any $r\in (0,R)$ let
\begin{align}\label{I(r)}
I(r)= \int_{B_r^+} \left(|\nabla U|^2+|\nabla V|^2+UV\right)dz-
\int_{B_r'} h(x)u(x)v(x)\, dx .
\end{align}
From
the fact that $U,V\in  H^1(B_R^+)$ and Lemma \ref{l:hardyball} it follows that $I\in
 W^{1,1}(0,R)$ and
\begin{equation} \label{I'(r)}
I'(r) =
\int_{S_r^+} \left(|\nabla U|^2+|\nabla V|^2+UV\right)\,dS-
\int_{\partial B_r'} h(x)u(x)v(x)\, dS'
\end{equation}
 for a.e. $r\in (0,R)$ and in the distributional sense.
Therefore $D\in W^{1,1}_{{\rm loc}}(0,R)$ and, replacing
\eqref{Pohozaev-2}, (\ref{I(r)}), and (\ref{I'(r)}) into
$D'(r)=r^{-N}[-(N-1)I(r)+rI'(r)]$,
we obtain (\ref{D'F}).
\end{proof}

In view of Lemma \ref{l:Hpos}, the function
\begin{equation}\label{eq:7}
\mathcal N:(0,r_0)\to\R,\quad
\mathcal N(r)=\frac{D(r)}{H(r)}
\end{equation}
is well defined.
As a consequence of estimate \eqref{eq:5} we obtain the following
corollary.
\begin{Corollary}\label{c:limpos}
    Let $(U,V)\in H^1(B_R^+)\times H^1(B_R^+)$ be as in Lemma
  \ref{l:Hpos}  and  let $D,H,\mathcal N$ be defined  in
  \eqref{eq:D(r)}, \eqref{eq:H(r)}, and \eqref{eq:7}
  respectively. For every $\e>0$ there exists $r_\e>0$ such that
\[
\mathcal N(r)+\e\geq 0\quad\text{for all }0<r<r_\e,
\]
i.e.
\begin{equation}\label{eq:liminf}
\liminf_{r\to0^+}\mathcal N(r)\geq0.
\end{equation}
\end{Corollary}

 \begin{Lemma}\label{mono} The function
   ${\mathcal N}$ defined in \eqref{eq:7} belongs to $W^{1,1}_{{\rm
       loc}}(0, r_0)$ and
\begin{align}\label{formulona}
{\mathcal N}'(r)=\nu_1(r)+\nu_2(r)
\end{align}
in a distributional sense and for a.e. $r\in (0,r_0)$,
where
\begin{align*}
\nu_1(r)=\frac{2r\Big[
    \left(\int_{S_r^+}
  \left(\left|\frac{\partial U}{\partial \nu}\right|^2+
\left|\frac{\partial V}{\partial \nu}\right|^2\right) dS\right) \cdot
    \left(\int_{S_r^+}
  (U^2+V^2)\,dS\right)-\left(
\int_{S_r^+}
  \left(U\frac{\partial U}{\partial \nu}+V\frac{\partial V}{\partial \nu}\right)\, dS\right)^{\!2} \Big]}
{\left(
\int_{S_r^+}
  (U^2+V^2)\,dS\right)^2}
\end{align*}
and
\begin{align}\label{eq:16}
  \nu_2(r)= &\, \frac{r\int_{S_r^+}  UV\,dS-2\int_{B_r^+} V\,\nabla
              U\cdot z\,dz-(N-1)\int_{B_r^+} UV\,dz}{\int_{S_r^+}
              (U^2+V^2)\,dS}\\
\notag&\quad+ \frac{(N-1)\int_{ B_r'}huv\, dx-r\int_{\partial B_r'}huv\,dS'+2\int_{B_r'}hu\,x\cdot
  \nabla_x v\,dx} {\int_{S_r^+}
    (U^2+V^2)\,dS}.
\end{align}
\end{Lemma}
\begin{proof}
  It follows directly from the definition of $\mathcal N$ and Lemmas
  \ref{l:hprime} and \ref{l:dprime}.
\end{proof}

We now estimate the term $\nu_2$ in \eqref{eq:16}. This is the most
delicate point in the development of the monotonicity argument for
system \eqref{eq:system}, due to the presence of  the
integral over ``the boundary of the boundary''
$\int_{\partial B_r'}huv\,dS'$ in the term $\nu_2$.

\begin{Lemma}\label{l:STIMA_DI_nu_2}
Let $\nu_2$ be as in \eqref{eq:16}. Then
\[
\nu_2(r)=O\left(1+\mathcal
  N(r)+r\sqrt{\frac{B(r)}{r^NH(r)}}\right)\quad\text{as }r\to 0^+,
\]
where
\begin{equation}\label{eq:B(r)}
B(r)=\int_{S_r^+}(|\nabla U|^2+|\nabla V|^2)\,dS.
\end{equation}
\end{Lemma}
\begin{proof}
We observe that
\begin{equation}\label{eq:10}
\frac{r\left|\int_{S_r^+}UV\,dS\right|}{\int_{S_r^+}(U^2+V^2)\,dS}=O(r)\quad\text{as
}r\to0^+.
\end{equation}
From \eqref{eq:8} and \eqref{eq:9} we have that
\begin{align}\label{eq:11}
\frac{\left|\int_{B_r^+} V\,\nabla
              U\cdot z\,dz\right|}{\int_{S_r^+}(U^2+V^2)\,dS}&\leq \frac{1}{2r^NH(r)}\left(\int_{B_r^+} V^2\,dz+
r^2\int_{B_r^+} |\nabla U|^2\,dz\right)\\
&\notag\leq
\frac{N+1}{2N}\mathcal N(r)r(1+O(r))+O(r)
\leq \mathcal N(r)r(1+O(r))+O(r)
\end{align}
and
\begin{equation}\label{eq:12}
\frac{\left|\int_{B_r^+} UV\,dz\right|}{\int_{S_r^+}
              (U^2+V^2)\,dS}\leq \frac{r}{2N}\mathcal N(r)(1+O(r)) +O(r)
          \end{equation}
as $r\to0^+$. From \eqref{eq:3} and \eqref{eq:8} we have that
\begin{equation}\label{eq:13}
  \frac{\left|\int_{ B_r'}huv\, dx\right|}
  {\int_{S_r^+}(U^2+V^2)\,dS}\leq
  \frac{\|h\|_{L^{\infty}(B_{r_0}')}}{2\widetilde \Lambda}\mathcal N(r)(1+O(r))+O(1)=\mathcal N(r)O(1)+O(1)
\end{equation}
as $r\to0^+$.

Integration by parts yields
\begin{equation*}
  \int_{ B_r'}hu\,x\cdot\nabla_xv\, dx=r \int_{\partial  B_r'}huv\, dS'
-\int_{ B_r'}v(Nhu+u\nabla h\cdot x+hx\cdot\nabla_xu)\, dx
\end{equation*}
so that
\begin{multline}\label{eq:14}
-r\int_{\partial B_r'}huv\,dS'+2\int_{B_r'}hu\,x\cdot
  \nabla_x v\,dx\\
=\int_{B_r'}hu\,x\cdot
  \nabla_x v\,dx-\int_{B_r'}hv\,x\cdot
  \nabla_x u\,dx-\int_{B_r'}uv(Nh+x\cdot
  \nabla h)\,dx.
\end{multline}
 From Lemma \ref{l:hardyball} and \eqref{eq:8} we have that
\begin{align}\label{eq:15}
  \frac{\left|\int_{B_r'}uv(Nh+x\cdot
  \nabla h)\,dx \right|}
  {\int_{S_r^+}(U^2+V^2)\,dS}&\leq
  \frac{\|Nh+x\cdot
  \nabla_x h\|_{L^{\infty}(B_{r_0}')}}{2\widetilde \Lambda}\mathcal
                               N(r)(1+O(r))+O(1)\\
&\notag=\mathcal N(r)O(1)+O(1)
\end{align}
as $r\to0^+$.

On the other hand, by the Divergence Theorem  we have that
\begin{align}\label{eq:18}
\int_{B_r'}hu\,x&\cdot
  \nabla_x v\,dx= -\int_{B_r'}hu\,(x\cdot
  \nabla_x v){\mathbf e}_{N+1}\cdot \nu\,dx \\
\notag&=\int_{S_r^+}h(x)U(x,t)\,(z\cdot
  \nabla V){\mathbf e}_{N+1}\cdot \nu\,dS-
\int_{B_r^+}\frac{\partial}{\partial t}\left[ h(x)U(x,t)\,(z\cdot
  \nabla V)\right]\,dz \\
\notag&=\int_{S_r^+}h(x)U(x,t)\,(z\cdot
  \nabla V){\mathbf e}_{N+1}\cdot \nu\,dS-
\int_{B_r^+}h(x)U_t\,(z\cdot
  \nabla V)\,dz\\
\notag&\qquad-
\int_{B_r^+}h(x)U\,(V_t+z\cdot
  \nabla V_t)\,dz\\
\notag&=\int_{S_r^+}h(x)U(x,t)\,(z\cdot
  \nabla V){\mathbf e}_{N+1}\cdot \nu\,dS-
\int_{B_r^+}h(x)U_t\,(z\cdot
  \nabla V)\,dz\\
\notag&\qquad-
r \int_{S_r^+}h(x)UV_t\,dS
+
\int_{B_r^+}(Nh(x)+\nabla h\cdot x )UV_t\,dz
+
\int_{B_r^+}
hV_t(\nabla U\cdot z)\,dz.
\end{align}
Hence, taking into account Lemma \ref{l:Poincare},
\begin{align} \label{eq:*}
  \left|\int_{B_r'}\!\!hu\,x\cdot
  \nabla_x v\,dx\right|\!\leq\!{\rm const\,} \!\!\left(
r\sqrt{r^N H(r)B(r)}
+r\!\!\int_{B_r^+}(|\nabla U|^2+|\nabla V|^2)\,dz\!+\!
\int_{S_r^+}(U^2+V^2)\,dS
\right)
\end{align}
for some ${\rm const\,>}0$ independent of $r$.
In a similar way we obtain that
\begin{align*}
  \left|\int_{B_r'}hv\,x\cdot
  \nabla_x u\,dx\right|\leq{\rm const\,} \left(
r\sqrt{r^N H(r)B(r)}
+r\int_{B_r^+}(|\nabla U|^2+|\nabla V|^2)\,dz+
\int_{S_r^+}(U^2+V^2)\,dS
\right).
\end{align*}
As a consequence, in view of \eqref{eq:8} we conclude that
\begin{equation}\label{eq:17}
\frac{\left|-r\int_{\partial B_r'}huv\,dS'+2\int_{B_r'}hu\,x\cdot
  \nabla_x v\,dx\right|} {\int_{S_r^+}
    (U^2+V^2)\,dS}\leq \mathcal N(r)O(1)+\sqrt{\frac{B(r)}{r^NH(r)}}\,O(r)+O(1)
\end{equation}
as $r\to0^+$.

Inserting \eqref{eq:10}-\eqref{eq:17} into \eqref{eq:16} the proof of the lemma follows.
\end{proof}

Inspired by
\cite[Lemma 5.9]{FF-edinb}, in the following lemma we estimate  $B$  in terms of the
derivative $D'$.

\begin{Lemma}\label{l:stima_B}
Let $B$ be defined in \eqref{eq:B(r)}. Then there exist $C_1,C_2,\bar
r>0$ such that
\[
B(r)\leq 2 r^{N-1}D'(r)+C_1 r^{N-2} (D(r)+C_2 H(r)) \quad\text{and}\quad
D(r)+C_2H(r)\geq 0\quad\text{for all }r\in
(0,\bar r).
\]
\end{Lemma}
\begin{proof}
From the definition of $D$ (see \eqref{eq:D(r)}) we have that
\begin{equation}\label{eq:21}
  D'(r)=r^{1-N}B(r)-(N-1)r^{-1}D(r)+r^{1-N}\int_{S_r^+}UV\,dS-r^{1-N}\int_{\partial B_r'}huv\,dS'.
\end{equation}
From \eqref{eq:14} it follows that
\begin{equation*}
\int_{\partial B_r'}huv\,dS'=\frac1r\int_{B_r'}hu\,x\cdot
  \nabla_x v\,dx+\frac1r \int_{B_r'}hv\,x\cdot
  \nabla_x u\,dx+\frac1r \int_{B_r'}uv(Nh+x\cdot
  \nabla h)\,dx.
\end{equation*}
 By \eqref{eq:*} and \eqref{eq:8} we deduce that, for
every $\e>0$, there exists $C_\e>0$ such that
\begin{align*}
\bigg|\frac1r\int_{B_r'}&hu\,x\cdot
  \nabla_x v\,dx\bigg|\\
&\leq \e B(r)+C_\e r^N H(r)+O(1) \int_{B_{r}^+} \left(|\nabla U|^2+|\nabla
V|^2\right)dz+O(1)r^{N-1}H(r)\\
&\leq \e B(r)+O(1)r^{N-1}H(r)+O(1)r^{N-1}D(r)\quad\text{as }r\to0^+.
\end{align*}
An analogous estimate holds for the term $\frac1r \int_{B_r'}hv\,x\cdot
  \nabla_x u\,dx$, whereas \eqref{eq:15} implies that
\[
\frac1r \int_{B_r'}uv(Nh+x\cdot
  \nabla h)\,dx =O(1)r^{N-1}H(r)+O(1)r^{N-1}D(r) \quad\text{as }r\to0^+.
\]
Therefore we conclude that
\begin{equation}\label{eq:20}
  \left|\int_{\partial B_r'}huv\,dS'\right|\leq  2\e B(r)+O(1)r^{N-1}H(r)+O(1)r^{N-1}D(r)\quad\text{as }r\to0^+.
\end{equation}
From Corollary \ref{c:limpos}, \eqref{eq:21} and \eqref{eq:20}, choosing $\e=\frac14$, we deduce
that, for some constants $C_1,C_2>0$ independent of $r$,
$D(r)+C_2H(r)\geq0$ and
\[
D'(r)\geq \frac12
r^{1-N}B(r)-\frac{C_1}2r^{-1}(D(r)+C_2H(r))\quad\text{for all $r$
  sufficiently small}.
\]
The proof is thereby complete.
\end{proof}

\begin{Lemma}\label{l:limitN}
Let $\mathcal N:(0,r_0)\to\R$ be defined in \eqref{eq:7}. Then
\begin{equation}\label{eq:23}
  \mathcal N(r)=O(1)\quad\text{as }r\to0^+.
\end{equation}
Furthermore
the limit
$$
\gamma:=\lim_{r\to 0^+}\mathcal N(r)
$$
exists, is finite and
\begin{equation*}
\gamma\geq 0.
\end{equation*}
\end{Lemma}
\begin{proof}
  Let us consider the set
\[
\Sigma=\{r\in(0,r_0):D'(r)H(r)\leq H'(r) D(r)\}
\]
(which is well-defined up to a zero measure set).

If there exists $r\in(0,r_0]$ such that $|(0,r)\cap\Sigma|_1=0$ (where
$|\cdot|_1$ stands for the Lebesgue measure in $\R$) we have that
$\mathcal N'\geq 0$ a.e. in $(0,r)$ and hence $\mathcal N$ is
non-decreasing in $(0,r)$ and admits a limit as $r\to0^+$ which is
necessarily finite and non-negative due to \eqref{eq:liminf}.

Let us now assume that, for all $r\in(0,r_0]$,
$|(0,r)\cap\Sigma|_1>0$. In view of Lemma \ref{l:stima_B} and \eqref{H'2} we have that,
a.e. in $(0,r_0)\cap \Sigma$,
\begin{align}\label{eq:22}
B(r)&\leq
2 r^{N-1}\frac{H'(r)D(r)}{H(r)}+C_1 r^{N-2}  (D(r)+C_2 H(r))\\
\notag&=4 r^{N-2}\frac{D^2(r)}{H(r)}+C_1  r^{N-2}  (D(r)+C_2 H(r)).
\end{align}
Schwarz inequality implies that the function $\nu_1$ appearing in
Lemma \ref{mono} is non-negative, hence \eqref{formulona},
Lemma \ref{l:STIMA_DI_nu_2}, and \eqref{eq:22} imply that
\[
\mathcal N'(r)\geq O(1)\left(
1+\mathcal N(r)+\sqrt{4 \mathcal N^2(r)+C_1(\mathcal
  N(r)+C_2)}\right)
\]
as $r\to0^+$, $r\in \Sigma$. Hence there exist $\tilde C,\tilde r>0$
such that
\[
\mathcal N'(r)\geq -\tilde C\left(
1+\mathcal N(r)\right)\quad\text{for a.e. }r\in(0,\tilde r)\cap\Sigma.
\]
Since the above inequality is obviously true in $(0,\tilde
r)\setminus\Sigma$ (provided $\tilde r$ is sufficiently small), we
deduce that
\begin{equation}\label{eq:25}
\mathcal N'(r)\geq -\tilde C\left(
1+\mathcal N(r)\right)\quad\text{for a.e. }r\in(0,\tilde r).
\end{equation}
Integrating the above inequality in $(r,\tilde r)$ we obtain that
\[
\mathcal N(r)+1\leq e^{\tilde C \tilde r} (\mathcal N(\tilde r)+1) \quad\text{for all }r\in(0,\tilde r).
\]
The above estimate together with Corollary \ref{c:limpos} yield
\eqref{eq:23}.
Furthermore \eqref{eq:25} implies that
\[
\left(e^{\tilde C r}(1+\mathcal N(r))\right)'\geq 0 \quad\text{a.e. in
}(0,\tilde r),
\]
hence the function $r\mapsto e^{\tilde C r}(1+\mathcal N(r))$ admits a
limit as $r\to0^+$. Therefore also the limit
$\gamma:=\lim_{r\to 0^+}\mathcal N(r)$ exists; furthermore $\gamma$ is
finite in view of \eqref{eq:23} and $\gamma\geq0$ in view of \eqref{eq:liminf}.
\end{proof}

A first consequence of the previous monotonicity argument is the
following estimate of the function $H$.
\begin{Lemma}\label{l:uppb}
Letting $\gamma$ be as in Lemma \ref{l:limitN}, we have that
\begin{equation}\label{eq:24}
H(r)=O(r^{2\gamma})\quad \text{as $r\to0^+$}.
\end{equation}
Furthermore, for any $\sigma>0$ there exist
$K(\sigma)>0$ depending on $\sigma$ such that
\begin{equation} \label{2ndest}
H(r)\geq K(\sigma)\,
  r^{2\gamma+\sigma} \quad \text{for all } r\in (0, r_0).
\end{equation}
\end{Lemma}

\begin{proof}  See the proof of \cite[Lemma 3.16]{FF}.
\end{proof}

\section{Blow-up analysis}\label{sec:blow-up-analysis}

\begin{Lemma}\label{l:blowup}
 Let $(U,V)\in H^1(B_R^+)\times H^1(B_R^+)$ be  a weak solution to
\eqref{eq:system} such that $(U,V)\not=(0,0)$,
  let $\mathcal N$ be defined  in  \eqref{eq:7}, and let $\gamma:=\lim_{r\rightarrow 0^+} {\mathcal
    N}(r)$ be as in Lemma \ref{l:limitN}. Then
\begin{itemize}
\item[\rm (i)] there exists $\ell\in \N$ such that
  $\gamma=\ell$;
\item[\rm (ii)] for every sequence $\lambda_n\to0^+$, there exist a subsequence
$\{\lambda_{n_k}\}_{k\in\N}$ and $2M_\ell$ real constants
$\beta_{\ell,m},\beta'_{\ell,m}$, $m=1,2,\dots,M_\ell$,  such that
$\sum_{m=1}^{M_\ell}((\beta_{\ell,m})^2+(\beta'_{\ell,m})^2)=1$
and
\[
\frac{U(\lambda_{n_k}z)}{\sqrt{H(\lambda_{n_k})}}\to
|z|^{\ell}
\sum_{m=1}^{M_\ell}\beta_{\ell,m}Y_{\ell,m}\Big(\frac
z{|z|}\Big),\quad
\frac{V(\lambda_{n_k}z)}{\sqrt{H(\lambda_{n_k})}}\to
|z|^{\ell}
\sum_{m=1}^{M_\ell}\beta'_{\ell,m}Y_{\ell,m}\Big(\frac
z{|z|}\Big),
\]
weakly in $H^1(B_1^+)$ and strongly in $H^1(B_r^+)$ for all
$r\in(0,1)$.
See Section \ref{s:introduction} for the definition of $M_\ell$ and $Y_{\ell,m}$.
\end{itemize}
\end{Lemma}
\begin{proof}
Let us define
\begin{equation}\label{eq:UVtau}
U_\lambda(z)=\frac{U(\lambda z)}{\sqrt{H(\lambda)}},
\quad
V_\lambda(z)=\frac{V(\lambda z)}{\sqrt{H(\lambda)}}.
\end{equation}
We notice that
\begin{equation}\label{eq:26}
\Delta U_\lambda=\lambda^2V_\lambda\quad\text{and}\quad \int_{S_1^+}(U_{\lambda}^2+V_{\lambda}^2)dS=1.
\end{equation}
By scaling and \eqref{eq:23} we have
\begin{multline}\label{eq:8bis}
  \int_{B_{1}^+} \left(|\nabla U_\lambda(z)|^2 +|\nabla V_\lambda(z)|^2+\lambda^2 U_\lambda(z) V_\lambda(z)\right)\,dz -
  \lambda
\int_{B_1'}h(\lambda
x)U_\lambda(x,0) V_\lambda(x,0)\,dx
  \\={\mathcal N}(\lambda)=O(1)
\end{multline}
as $\lambda\to0^+$.
On the other hand, Lemmas \ref{l:hardyball} and \ref{l:Poincare} imply
\begin{multline*}
  \mathcal N(\lambda)\geq
\left(\int_{B_{1}^+} \left(|\nabla U_\lambda(z)|^2 +|\nabla V_\lambda(z)|^2\right)\,dz
\right)\left(1-\frac{\lambda^2}{2N}-
\frac{\lambda\|h\|_{L^{\infty}(B_{r_0}')}}{2\widetilde \Lambda}\right)\\
-\frac{\lambda^2}{2N}-\frac{\lambda\|h\|_{L^{\infty}(B_{r_0}')}(N-1)}{4\widetilde \Lambda}
\end{multline*}
so that \eqref{eq:8bis} and Lemma \ref{l:Poincare} imply
that
\begin{equation*}
  \{U_\lambda\}_{\lambda\in(0,\tilde \lambda)}
\text{ and }\{V_\lambda\}_{\lambda\in(0,\tilde \lambda)}
\text{ are bounded in }H^1(B_1^+)
\end{equation*}
for some $\tilde\lambda>0$.

Therefore, for any given sequence $\lambda_n\to 0^+$, there exists a
subsequence $\lambda_{n_k}\to0^+$ such that $U_{\lambda _{n_k}}\weakly \widetilde U$
and $V_{\lambda _{n_k}}\weakly\widetilde V$
 weakly in $H^1(B_1^+)$ for some $\widetilde U,\widetilde V\in H^1(B_1^+)$.
 From compactness of the trace embedding $H^1(B_1^+)\hookrightarrow L^2(S_1^+)$ and from \eqref{eq:26} we
deduce that
\begin{equation}\label{eq:27}
\int_{S_1^+}(\widetilde U^2+\widetilde V^2)dS=1,
\end{equation}
hence $(\widetilde U,\widetilde V)\neq(0,0)$, i.e. $\widetilde U$ and
$\widetilde V$ can not  both vanish identically.
For every  $\lambda\in (0,\tilde \lambda)$, the couple $(U_\lambda,V_\lambda)$ satisfies
\begin{equation} \label{eqlam}
\begin{cases}
   \Delta U_\lambda=\lambda^2 V_\lambda,&\text{in }B_1^+,\\
   \Delta V_\lambda=0,&\text{in }B_1^+,\\
\partial_\nu U_\lambda=0, &\text{on }B_1',\\
\partial_\nu V_\lambda=\lambda h(\lambda x)u_\lambda, &\text{on }B_1',
\end{cases}
\end{equation}
in a weak sense, i.e.
\begin{equation}\label{eq:8tau}
\begin{cases}
\int_{B_1^+}\nabla
U_\lambda\cdot\nabla \varphi\, dz=-\lambda^2
\int_{B_1^+}V_\lambda\varphi\, dz \\
\int_{B_1^+}\nabla
V_\lambda\cdot\nabla \varphi\, dz=\lambda
\int_{B_1'}h(\lambda x)u_\lambda(x)\mathop{\rm Tr}\varphi(x) \,dx,
\end{cases}
\end{equation}
for all $\varphi \in H^1(B_1^+)$  such that
$\varphi=0\text{ on }S_1^+$,
where $u_\lambda=\mathop{\rm Tr}U_\lambda$.
From the weak convergences  $U_{\lambda _{n_k}}\weakly \widetilde U$
and $V_{\lambda _{n_k}}\weakly\widetilde V$ in $H^1(B_1^+)$, we can pass to
the limit in \eqref{eq:8tau} to obtain
\begin{equation*}
\begin{cases}
\int_{B_1^+}\nabla
\widetilde U\cdot\nabla \varphi\,dz=0,\\
\int_{B_1^+}\nabla \widetilde
V\cdot\nabla \varphi\,dz=0,
\end{cases}\quad \text{for all $\varphi \in H^1(B_1^+)$  such that
$\varphi=0\text{ on }S^1_+$},
\end{equation*}
i.e. $(\widetilde U,\widetilde V)$ weakly solves
\begin{equation}\label{eq:extended_limit}
\begin{cases}
   \Delta \widetilde U=0,&\text{in }B_1^+,\\
   \Delta \widetilde V=0,&\text{in }B_1^+,\\
\partial_\nu \widetilde U=0, &\text{on }B_1',\\
\partial_\nu \widetilde V=0, &\text{on }B_1'.
\end{cases}
\end{equation}
From elliptic regularity under Neumann boundary conditions (see in
particular \cite[Theorem 8.13]{Salsa}) we conclude that
\begin{equation}\label{eq:30}
  \{U_\lambda\}_{\lambda\in(0,\tilde \lambda)} \text{ and }
\{V_\lambda\}_{\lambda\in(0,\tilde \lambda)} \text{ are bounded in
$H^2(B_r^+)$ for all $r\in(0,1)$},
\end{equation}
hence, by compactness, up to passing
to a subsequence,
\begin{equation}\label{eq:strong_conv}
U_{\lambda _{n_k}}\to \widetilde U\ \text{and}\
V_{\lambda _{n_k}}\to\widetilde V\ \text{
 weakly in $H^2(B_r^+)$ and strongly in $H^1(B_r^+)$ for all $r\in(0,1)$}.
\end{equation}
For any $r\in (0,1)$ and $k\in \N$, let us define
the functions
\begin{align*}
&D_k(r)=r^{-N+1} \bigg[\int_{B_r^+} \left(|\nabla
  U_{\lambda_{n_k}}|^2+|\nabla V_{\lambda_{n_k}}|^2+\lambda_{n_k}^2 U_{\lambda_{n_k}}\!
  V_{\lambda_{n_k}}
\right)\,dz\\
&\hskip6cm-\lambda_{n_k}\int_{B_r'} h(\lambda_{n_k} x)u_{\lambda_{n_k}}\!(x)v_{\lambda_{n_k}}\!(x)\, dx  \bigg],\\
&H_k(r)=r^{-N}\int_{S_r^+} (U_{\lambda_{n_k}}^2+V_{\lambda_{n_k}}^2)\, dS,
\end{align*}
where we have set $v_\lambda=\mathop{\rm Tr}V_\lambda$.
By direct calculations we have
\begin{equation}\label{NkNw}
{\mathcal
    N}_k(r):=\frac{D_k(r)}{H_k(r)}=\frac{D(\lambda_{n_k}r)}{H(\lambda_{n_k}r)}
  ={\mathcal N}(\lambda_{n_k}r) \quad \text{for all } r\in (0,1).
\end{equation}
From \eqref{eq:strong_conv} it follows that,
for any fixed $r\in (0,1)$,
\begin{equation} \label{convDk}
 D_k(r)\to \widetilde D(r)\quad\text{and}\quad H_k(r)\to \widetilde
 H(r)\quad\text{as }k\to+\infty
\end{equation}
where
\begin{equation} \label{Dw(r)}
  \widetilde D (r)=
r^{-N+1} \int_{B_r^+} \left(|\nabla
  \widetilde U|^2+|\nabla \widetilde V|^2
\right)dz\quad\text{and}\quad
\widetilde H (r)=
r^{-N}\int_{S_r^+} (\widetilde U^2+\widetilde V^2)\, dS
\end{equation}
for all $r\in(0,1)$.
We observe that $\widetilde H (r)>0$ for all $r\in(0,1)$; indeed, if
$\widetilde H (\bar r)=0$ for some $\bar r\in(0,1)$, the fact that $\widetilde
U,\widetilde V$ (and their even extension for $t<0$) are harmonic would imply that $\widetilde
U\equiv\widetilde V\equiv 0$ in $B_{\bar r}^+$, thus contradicting the
classical unique continuation principle.
Therefore the function
\[
\widetilde {\mathcal N}(r):=\frac{\widetilde  D(r)}{\widetilde  H(r)}
\]
is well defined for $r\in (0,1)$.
From (\ref{NkNw}), (\ref{convDk}), and
Lemma \ref{l:limitN}, we deduce  that
\begin{equation}\label{Nw(r)}
\widetilde {\mathcal    N}(r)=\lim_{k\to \infty} {\mathcal N}(\lambda_{n_k}r)=\gamma
\end{equation}
for all $r\in (0,1)$.
Therefore $\widetilde {\mathcal N}$ is  constant in $(0,1)$ and hence $\widetilde {\mathcal
  N}'(r)=0$ for any $r\in (0,1)$.
Arguing as in the proof of  Lemma \ref{mono} we can prove that
\begin{equation*}
\widetilde{\mathcal N}'(r)=
\frac{2r\Big[
    \Big(\int_{S_r^+}
  \Big(\Big|\frac{\partial \widetilde U}{\partial
        \nu}\Big|^2+\Big|\frac{\partial \widetilde V}
{\partial \nu}\Big|^2\Big) dS\Big) \cdot
    \left(\int_{S_r^+}
  (\widetilde U^2+\widetilde V^2)\,dS\right)-\left(
\int_{S_r^+}
  \left(\widetilde U\frac{\partial \widetilde U}{\partial
      \nu}+\widetilde V\frac{\partial \widetilde V}
{\partial \nu}\right)\, dS\right)^{\!2} \Big]}
{\left(
\int_{S_r^+}
  (\widetilde U^2+\widetilde V^2)\,dS\right)^2}
\end{equation*}
for all $r\in (0,1)$. Therefore for all $r\in (0,1)$
\begin{equation*}
    \Big(\int_{S_r^+}
  \Big(\Big|\tfrac{\partial \widetilde U}{\partial
        \nu}\Big|^2+\Big|\tfrac{\partial \widetilde V}
{\partial \nu}\Big|^2\Big) dS\Big) \cdot
    \left(\int_{S_r^+}
  (\widetilde U^2+\widetilde V^2)\,dS\right)-\left(
\int_{S_r^+}
  \left(\widetilde U\tfrac{\partial \widetilde U}{\partial
      \nu}+\widetilde V\tfrac{\partial \widetilde V}
{\partial \nu}\right)\, dS\right)^{\!2}=0
\end{equation*}
which implies that $(\widetilde U,\widetilde V)$ and
$(\frac{\partial \widetilde U}{\partial \nu},\frac{\partial \widetilde
  V}{\partial \nu})$
have the same direction as vectors in $L^2(S_r^+)\times L^2(S_r^+)$. Hence
there exists a function $\eta=\eta(r)$ such that
$\left(\frac{\partial \widetilde U}{\partial \nu}(r\theta),\frac{\partial \widetilde
  V}{\partial \nu}(r\theta)\right)=\eta(r) (\widetilde U (r\theta),\widetilde V (r\theta))$
for all $r\in(0,1)$ and $\theta\in {\mathbb S}^N_+$.  By
integration we obtain
\begin{align}
\label{separate}
&\widetilde U(r\theta)=e^{\int_1^r \eta(s)ds} \widetilde U(\theta)
=\varphi(r) \Psi_1(\theta), \quad  r\in(0,1), \ \theta\in {\mathbb S}^N_+,\\
\label{separate2}
&\widetilde V(r\theta)=e^{\int_1^r \eta(s)ds} \widetilde V(\theta)
=\varphi(r) \Psi_2(\theta), \quad  r\in(0,1), \ \theta\in {\mathbb S}^N_+,\end{align}
where $\varphi(r)=e^{\int_1^r \eta(s)ds}$ and
$\Psi_1=\widetilde U\big|_{{\mathbb S}^{N}_+}$, $\Psi_2(\theta)=\widetilde V\big|_{{\mathbb S}^{N}_+}$.
From \eqref{eq:extended_limit},  \eqref{separate}, and \eqref{separate2},  it follows that
\begin{equation}\label{eq:28}
\begin{cases}
r^{-N}\big(r^{N}\varphi'\big)'\Psi_i(\theta)+
r^{-2}\varphi(r)\Delta_{{\mathbb S}^{N}_+}\Psi_i(\theta)
=0,&\text{on }{\mathbb S}^{N}_+,\\
\partial_\nu \Psi_i=0,&\text{on }\partial{\mathbb S}^{N}_+,
\end{cases}\quad i=1,2.
\end{equation}
Taking $r$ fixed we deduce that $\Psi_1,\Psi_2$ are either zero or
restrictions to ${\mathbb S}^{N}_+$ of eigenfunctions of $-\Delta_{{\mathbb S}^{N}}$  associated to the same
eigenvalue and symmetric with respect to the equator $\partial{\mathbb S}^{N}_+$. Therefore there
exist $\ell\in\N$, $\{\beta_{\ell,m},\beta'_{\ell,m}\}_{m=1}^{M_\ell}\subset \R$ such that
\[
\begin{cases}
-\Delta_{{\mathbb S}^{N}_+}\Psi_1=\lambda_\ell \Psi_1,&\text{on }{\mathbb S}^{N}_+,\\
\partial_\nu \Psi_1=0,&\text{on }\partial{\mathbb S}^{N}_+,
\end{cases}
\quad
\begin{cases}
-\Delta_{{\mathbb S}^{N}_+}\Psi_2=\lambda_\ell \Psi_2,&\text{on }{\mathbb S}^{N}_+,\\
\partial_\nu \Psi_2=0,&\text{on }\partial{\mathbb S}^{N}_+,
\end{cases}
\]
and
\[
\Psi_1=\sum_{m=1}^{M_\ell}\beta_{\ell,m}Y_{\ell,m},\quad
\Psi_2=\sum_{m=1}^{M_\ell}\beta'_{\ell,m}Y_{\ell,m}.
\]
In view of \eqref{eq:27} we have that $\int_{{\mathbb
    S}^{N}_+}(\Psi_1^2+\Psi_2^2)\,dS=1$ and hence
\[
\sum_{m=1}^{M_\ell}((\beta_{\ell,m})^2+(\beta'_{\ell,m})^2)=1.
\]
Since $\Psi_1$ and $\Psi_2$ are not both identically zero, from
\eqref{eq:28} it follows that $\varphi(r)$ solves the equation
\[
\varphi''(r)+\frac{N}r\varphi'(r)-\frac{\lambda_\ell}{r^2}\varphi(r)=0
\]
and hence $\varphi(r)$ is of the form
$$
\varphi(r)=c_1 r^{\ell}+c_2 r^{-(N-1)-\ell}
$$
for some $c_1,c_2\in\R$.
Since either
$|z|^{-(N-1)-\ell}\Psi_1(\frac{z}{|z|}) \notin
H^1(B_1^+)$ or $|z|^{-(N-1)-\ell}\Psi_2(\frac{z}{|z|}) \notin
H^1(B_1^+)$ (being $(\Psi_1,\Psi_2)\not\equiv(0,0)$), we have  that $c_2=0$ and  $\varphi(r)=c_1
r^{\ell}$. Moreover, from $\varphi(1)=1$ we deduce  that $c_1=1$. Then
\begin{equation} \label{expw}
\widetilde U(r\theta)=r^{\ell} \Psi_1(\theta),\quad
\widetilde V(r\theta)=r^{\ell} \Psi_2(\theta),  \quad
\text{for all }r\in (0,1)\text{ and }\theta\in {\mathbb S}^N_+.
\end{equation}
From \eqref{expw} and the fact that
\[
\int_{{\mathbb
    S}^{N}_+}(\Psi_1^2+\Psi_2^2)\,dS=1\quad\text{and}\quad
\int_{{\mathbb
    S}^{N}_+}(|\nabla_{{\mathbb
        S}^{N}}\Psi_1|^2+|\nabla_{{\mathbb
        S}^{N}}\Psi_2|^2)\,dS=\lambda_\ell
\]
 it follows that
\begin{align*}
 \widetilde D (r)&=
\frac{1}{r^{N-1}} \int_{B_r^+}
(|\nabla \widetilde U|^2+|\nabla \widetilde V|^2)\,dt\,dx  \\
&=r^{1-N}\ell^2\int_0^rt^{N+2(\ell-1)}dt+
r^{1-N}\lambda_\ell\int_0^r t^{N+2(\ell-1)}dt=\frac{\ell^2+\ell(N-1+\ell)}{N+2\ell-1}\,r^{2\ell}=
\ell \,r^{2 \ell}
\end{align*}
and
\begin{align*}
  \widetilde  H(r)=\int_{{\mathbb
      S}^{N}_+}\left(\widetilde U^2(r\theta)+\widetilde V^2(r\theta)\right)\,dS=r^{2 \ell}.
\end{align*}
Hence from (\ref{Nw(r)}) it follows
that $\gamma=\widetilde{\mathcal N}(r)=\frac{\widetilde D(r)}{\widetilde H(r)}=\ell$.
The proof of the lemma is  complete.
\end{proof}

\begin{Lemma} \label{l:limite}
 Let $(U,V)\in H^1(B_R^+)\times H^1(B_R^+)$ be  a weak solution to
\eqref{eq:system} such that $(U,V)\not=(0,0)$,
  let $H$ be defined  in  \eqref{eq:H(r)}, and let $\ell$ be as in
  Lemma \ref{l:blowup}. Then the limit
\[
\lim_{r\to0^+}r^{-2\ell}H(r)
\]
exists and it is finite.
\end{Lemma}
\begin{proof}
We recall from Lemma \ref{l:blowup} that
$\ell=\lim_{r\rightarrow 0^+} {\mathcal
    N}(r)$ with $\mathcal N$ as in \eqref{eq:7}.

In view of \eqref{eq:24} it is sufficient to prove that the limit
exists. By \eqref{H'2}
and Lemma~\ref{l:limitN} we have
\begin{align}\label{eq:31}
\frac{d}{dr} \frac{H(r)}{r^{2\ell}}& =-2\ell r^{-2\ell-1}
H(r)+r^{-2\ell} H'(r) =2r^{-2\ell-1} (D(r)-\ell
H(r))\\
\notag&=2r^{-2\ell-1} H(r) \int_0^r {\mathcal N}'(\rho) d\rho.
\end{align}
From \eqref{eq:25} and \eqref{eq:23} it follows that there exists some
$c>0$ such that $\mathcal N'(r)\geq -c$ for all $r\in(0,\tilde
r)$. Then we can write $\mathcal N'(r)=-c+f(r)$ for some function
$f\in L^1_{\rm loc}(0,r_0)$
such that $f(r)\geq 0$ a.e. in $(0,\tilde r)$. Then
integration of \eqref{eq:31} over $(r,\tilde r)$ yields
\begin{equation}\label{inte}
  \frac{H(\tilde r)}{\tilde r^{2\ell}}-
  \frac{H(r)}{r^{2\ell}}=2\int_r^{\tilde r} \rho^{-2\ell-1}
  H(\rho) \left( \int_0^\rho f(t) dt \right) d\rho -2c\int_r^{\tilde r} \rho^{-2\ell}
  H(\rho) d\rho.
\end{equation}
Since $f\geq0$, we have that
$\lim_{r\to 0^+} \int_r^{\tilde r}  \rho^{-2\ell-1} H(\rho) \left( \int_0^\rho
  f(t) dt \right) d\rho$
exists.  On the other hand, \eqref{eq:24} implies that $\rho^{-2\ell}
H(\rho) \in L^1(0,\widetilde r)$.
Therefore both terms at the right hand side of
(\ref{inte}) admit a limit as $r\to 0^+$ (one of which is finite)
and the proof is complete.
\end{proof}

 Let $(U,V)\in H^1(B_R^+)\times H^1(B_R^+)$ be  a weak solution to
\eqref{eq:system} such that $(U,V)\not=(0,0)$.    Let us expand $U$ and $V$  as
\begin{equation*}
U(z)=U(\lambda
\theta)=\sum_{k=0}^\infty\sum_{m=1}^{M_k}\varphi_{k,m}(\lambda)Y_{k,m}(\theta),\quad
V(z)=V(\lambda \theta)=\sum_{k=0}^\infty\sum_{m=1}^{M_k}\widetilde\varphi_{k,m}(\lambda)Y_{k,m}(\theta)
\end{equation*}
where $\lambda=|z|\in(0,R]$, $\theta=z/|z|\in{{\mathbb S}^{N}_+}$, and
\begin{equation}\label{eq:37}
  \varphi_{k,m}(\lambda)=\int_{{\mathbb S}^{N}_+}U(\lambda\,\theta) Y_{k,m}(\theta)\,dS,\quad
\widetilde\varphi_{k,m}(\lambda)=\int_{{\mathbb S}^{N}_+}V(\lambda\,\theta)
    Y_{k,m}(\theta)\,dS
\end{equation}

\begin{Lemma} \label{l:coeff_fourier}
 Let $(U,V)\in H^1(B_R^+)\times H^1(B_R^+)$ be  a weak solution to
\eqref{eq:system} such that $(U,V)\not=(0,0)$,
  let $\ell$ be as in
  Lemma \ref{l:blowup}, and let $\widetilde\varphi_{\ell,m},
  \varphi_{\ell,m}$ be as in \eqref{eq:37}. Then, for all $1\leq m\leq
  M_\ell$,
\begin{align}
\label{eq:24b}\varphi_{\ell,m}(\lambda)&=
\lambda^{\ell}
\bigg(c_1^{\ell,m}+\int_\lambda^R\frac{t^{-\ell+1}}{2\ell+N-1}\widetilde
                 \varphi_{\ell,m}(t)\,dt\bigg)+
\lambda^{-(N-1)-\ell}\int_0^\lambda\frac{t^{N+\ell}}{N+2\ell-1}
\widetilde \varphi_{\ell,m}(t)\,dt\\
\notag&=
\lambda^{\ell}
\bigg(c_1^{\ell,m}+\int_\lambda^R\frac{t^{-\ell+1}}{2\ell+N-1}\widetilde
                 \varphi_{\ell,m}(t)\,dt+O(\lambda^2)\bigg),\quad\text{as
                    }\lambda\to0^+,\\
\label{eq:24c}\widetilde\varphi_{\ell,m}(\lambda) &=
\lambda^{\ell}
\bigg(d_1^{\ell,m}+\int_\lambda^R\frac{t^{-\ell+1}}{2\ell+N-1}\zeta_{\ell,m}(t)\,dt\bigg)+
\lambda^{-(N-1)-\ell}\int_0^\lambda\frac{t^{N+\ell}}{N+2\ell-1}
\zeta_{\ell,m}(t)\,dt\\
\notag&=\lambda^{\ell}
\bigg(d_1^{\ell,m}+\int_\lambda^R\frac{t^{-\ell+1}}{2\ell+N-1}\zeta_{\ell,m}(t)\,dt+O(\lambda)\bigg),\quad\text{as
                    }\lambda\to0^+,
\end{align}
where
\begin{equation}\label{eq:38ell}
  \zeta_{\ell,m}(\lambda)=\frac{1}{\lambda}\int_{{\mathbb
      S}^{N-1}}h(\lambda\theta')U(\lambda\theta',0)Y_{\ell,m}(\theta',0)\,dS',
\end{equation}
and
\begin{align}
\label{eq:36}&  c_1^{\ell,m}=R^{-\ell}\int_{{\mathbb S}^{N}_+}U(R\,\theta)
  Y_{\ell,m}(\theta)\,dS-
\frac{R^{-N-2\ell+1}}{N+2\ell-1}\int_0^R t^{N+\ell}\left(\int_{{\mathbb S}^{N}_+}V(t\,\theta)
    Y_{\ell,m}(\theta)\,dS\right)\,dt,\\
\label{eq:39}&  d_1^{\ell,m}=R^{-\ell}\int_{{\mathbb S}^{N}_+}V(R\,\theta)
  Y_{\ell,m}(\theta)\,dS\\
\notag&\qquad\qquad-
\frac{R^{-N-2\ell+1}}{N+2\ell-1}\int_0^R t^{N+\ell-1}\left(
\int_{{\mathbb
      S}^{N-1}}h(t\theta')U(t\theta',0)Y_{k,m}(\theta',0)\,dS'\right)\,dt.
\end{align}
\end{Lemma}
\begin{proof}
  From the Parseval identity it follows that
\begin{equation}\label{eq:17bis}
H(\lambda)=\int_{{\mathbb
    S}^{N}_+}\big(U^2(\lambda\theta)+V^2(\lambda\theta)\big)\,dS=
\sum_{k=0}^\infty\sum_{m=1}^{M_k}\big(\varphi_{k,m}^2(\lambda)+\widetilde\varphi_{k,m}^2(\lambda)\big),
\quad\text{for all }0<\lambda\leq R.
\end{equation}
In particular \eqref{eq:24} and \eqref{eq:17bis} yield, for all
$k\geq0$ and $1\leq m\leq M_k$,
\begin{equation}\label{eq:23b}
\varphi_{k,m}(\lambda)=O(\lambda^{\ell})\quad\text{and}\quad
\widetilde\varphi_{k,m}(\lambda)=O(\lambda^{\ell})
\quad\text{as }\lambda\to0^+.
\end{equation}
Equations \eqref{eq:system} and \eqref{eq:sph_eig} imply that, for every $k\geq0$ and $1\leq m\leq M_k$,
\begin{equation*}
\begin{cases}
-\varphi_{k,m}''(\lambda)-\frac{N}{\lambda}\varphi_{k,m}'(\lambda)+
\frac{k(N-1+k)}{\lambda^2} \, \varphi_{k,m}(\lambda)=\widetilde\varphi_{k,m}(\lambda),&\text{in
}(0,R),\\[10pt]
-\widetilde\varphi_{k,m}''(\lambda)-\frac{N}{\lambda}\widetilde\varphi_{k,m}'(\lambda)+
 \frac{k(N-1+k)}{\lambda^2} \, \widetilde\varphi_{k,m}(\lambda)=\zeta_{k,m}(\lambda),&\text{in
}(0,R),
\end{cases}
\end{equation*}
where
\begin{equation}\label{eq:38}
  \zeta_{k,m}(\lambda)=\frac{1}{\lambda}\int_{{\mathbb
      S}^{N-1}}h(\lambda\theta')U(\lambda\theta',0)Y_{k,m}(\theta',0)\,dS'.
\end{equation}
By direct calculations we have, for some $c_1^{k,m},c_2^{k,m},d_1^{k,m},d_2^{k,m}\in\R$,
\begin{align}
\label{eq:33}\varphi_{k,m}(\lambda)=\lambda^{k}
\bigg(c_1^{k,m}&+\int_\lambda^R\frac{t^{-k+1}}{2k+N-1}\widetilde
                 \varphi_{k,m}(t)\,dt\bigg)\\
\notag&+\lambda^{-(N-1)-k}
\bigg(c_2^{k,m}+\int_\lambda^R\frac{t^{N+k}}{1-N-2k}
\widetilde \varphi_{k,m}(t)\,dt\bigg),\\
\label{eq:34}\widetilde\varphi_{k,m}(\lambda)=\lambda^{k}
\bigg(d_1^{k,m}&+\int_\lambda^R\frac{t^{-k+1}}{2k+N-1}\zeta_{k,m}(t)\,dt\bigg)\\
\notag&+\lambda^{-(N-1)-k}
\bigg(d_2^{k,m}+\int_\lambda^R\frac{t^{N+k}}{1-N-2k}
\zeta_{k,m}(t)\,dt\bigg).
\end{align}
We observe that
\begin{equation}\label{eq:32}
\zeta_{k,m}(\lambda)=\frac{2^{N-1}\sqrt{H(2\lambda)}}{\lambda}
\int_{\partial B'_{1/2}}h(2\lambda x)U_{2\lambda}(x,0)Y_{k,m}\big(\tfrac{x}{|x|},0\big)\,dS'
\end{equation}
with $U_\lambda$ as in \eqref{eq:UVtau}. Since $\{U_\lambda\}_\lambda$
is bounded in $H^2(B^+_{1/2})$ in view of \eqref{eq:30}, from
continuity of the trace embedding $H^2(B^+_{1/2})\hookrightarrow
H^{3/2}(B_{1/2}')$ we deduce that  $\{\mathop{\rm Tr} U_\lambda\}_\lambda$
is bounded in $H^1(B_{1/2}')$ and its trace on $\partial B_{1/2}'$ is
bounded in $L^2(\partial B'_{1/2})$. Hence from \eqref{eq:32} and
\eqref{eq:24} we conclude that, for all
$k\geq0$ and $1\leq m\leq M_k$,
\begin{equation}\label{eq:zeta}
 \zeta_{k,m}(\lambda)=O(\lambda^{\ell-1})\quad\text{as }\lambda\to 0^+.
\end{equation}
From \eqref{eq:23b} and \eqref{eq:zeta} it follows that, for all $1\leq
m\leq M_\ell$, the functions
\[
t\mapsto t^{-\ell+1}\widetilde
                 \varphi_{\ell,m}(t),
\quad
t\mapsto t^{N+\ell}\widetilde \varphi_{\ell,m}(t),\quad
t\mapsto t^{-\ell+1}\zeta_{\ell,m}(t),
\quad
t\mapsto t^{N+\ell}\zeta_{\ell,m}(t),
\]
belong to $L^1(0,R)$. Hence
\begin{align*}
  &\lambda^{\ell}
    \bigg(c_1^{\ell,m}+\int_\lambda^R\frac{t^{-\ell+1}}{2\ell+N-1}\widetilde
    \varphi_{\ell,m}(t)\,dt\bigg)=o(\lambda^{-(N-1)-\ell}),\quad\text{as
    }\lambda\to0^+,\\
  &\lambda^{\ell}
    \bigg(d_1^{\ell,m}+\int_\lambda^R\frac{t^{-\ell+1}}{2\ell+N-1}\zeta_{\ell,m}(t)\,dt\bigg)=o(\lambda^{-(N-1)-\ell}),
    \quad\text{as }\lambda\to0^+,
\end{align*}
and consequently, by \eqref{eq:23b}, there must be
\[
c_2^{\ell,m}=-\int_0^R\frac{t^{N+\ell}}{1-N-2\ell}
\widetilde \varphi_{\ell,m}(t)\,dt\quad\text{and}\quad
d_2^{\ell,m}=-\int_0^R\frac{t^{N+\ell}}{1-N-2\ell}
\zeta_{\ell,m}(t)\,dt.
\]
Using \eqref{eq:23b} and \eqref{eq:zeta}, we then deduce that
\begin{align}\label{eq:12-alt}
&\lambda^{-(N-1)-\ell}
\bigg(c_2^{\ell,m}+\int_\lambda^R\tfrac{t^{N+\ell}}{1-N-2\ell}
\widetilde \varphi_{\ell,m}(t)\,dt\bigg)
=\lambda^{-(N-1)-\ell}\int_0^\lambda\tfrac{t^{N+\ell}}{N+2\ell-1}
\widetilde \varphi_{\ell,m}(t)\,dt
=O(\lambda^{\ell+2}),\\
\label{eq:12b}&\lambda^{-(N-1)-\ell}
\bigg(d_2^{\ell,m}+\int_\lambda^R\tfrac{t^{N+\ell}}{1-N-2\ell}
\zeta_{\ell,m}(t)\,dt\bigg)
=\lambda^{-(N-1)-\ell}\int_0^\lambda\tfrac{t^{N+\ell}}{N+2\ell-1}
\zeta_{\ell,m}(t)\,dt
=O(\lambda^{\ell+1}),
  \end{align}
as $\lambda\to0^+$.  From \eqref{eq:33}, \eqref{eq:34}, \eqref{eq:12-alt},
and\eqref{eq:12b} we deduce \eqref{eq:24b} and \eqref{eq:24c}.
Finally, \eqref{eq:36} and \eqref{eq:39} follow by computing
\eqref{eq:24b} and \eqref{eq:24c} for $\lambda=R$ and recalling \eqref{eq:37}.
\end{proof}

\noindent We now prove that
$\lim_{r\to 0^+} r^{-2\ell} H(r)$ is strictly positive.

\begin{Lemma} \label{l:limitepositivo} Under the same assumption as in
  Lemmas \ref{l:limite}, we have
\[
\lim_{r\to0^+}r^{-2\ell}H(r)>0.
\]
\end{Lemma}
\begin{proof}
Let us assume by contradiction that
$\lim_{\lambda\to0^+}\lambda^{-2\ell}H(\lambda)=0$. Then,
for all $1\leq m\leq M_\ell$, (\ref{eq:17bis})
would imply that
\begin{equation*}
  \lim_{\lambda\to0^+}\lambda^{-\ell}\varphi_{\ell,m}(\lambda)=
\lim_{\lambda\to0^+}\lambda^{-\ell}\widetilde\varphi_{\ell,m}(\lambda)=0.
\end{equation*}
Hence, in view of \eqref{eq:24b} and \eqref{eq:24c},
\[
c_1^{\ell,m}+\int_0^R\frac{t^{-\ell+1}}{2\ell+N-1}\widetilde
                 \varphi_{\ell,m}(t)\,dt=0\quad\text{and}\quad
d_1^{\ell,m}+\int_0^R\frac{t^{-\ell+1}}{2\ell+N-1}\zeta_{\ell,m}(t)\,dt=0
\]
which, in view of (\ref{eq:zeta}) and \eqref{eq:23b}, yields
\begin{align}\label{eq:13tris}
\lambda^{\ell}
\bigg(c_1^{\ell,m}+\int_\lambda^R\frac{t^{-\ell+1}}{2\ell+N-1}\widetilde
                 \varphi_{\ell,m}(t)\,dt\bigg)=
\lambda^{\ell}\int_0^\lambda\frac{t^{-\ell+1}}{1-2\ell-N}\widetilde
                 \varphi_{\ell,m}(t)\,dt=O(\lambda^{\ell+2})\\
\label{eq:13q}
\lambda^{\ell}
\bigg(d_1^{\ell,m}+\int_\lambda^R\frac{t^{-\ell+1}}{2\ell+N-1}\zeta_{\ell,m}(t)\,dt\bigg)=
\lambda^{\ell}\int_0^\lambda\frac{t^{-\ell+1}}{1-2\ell-N}\zeta_{\ell,m}(t)\,dt=O(\lambda^{\ell+1})
\end{align}
as $\lambda\to0^+$. Estimates \eqref{eq:24b}, \eqref{eq:24c},
\eqref{eq:13tris}, and \eqref{eq:13q} imply that
\[
\varphi_{\ell,m}(\lambda)=O(\lambda^{\ell+2})\quad\text{and}\quad
\widetilde\varphi_{\ell,m}(\lambda)=O(\lambda^{\ell+1})\quad\text{as }\lambda\to0^+
\quad\text{for every }
1\leq m\leq M_\ell,
\]
namely,
\[
\sqrt{H(\lambda)}\,(U_\lambda,Y_{\ell,m})_{L^2({\mathbb
  S}^{N}_+)}=
O(\lambda^{\ell+2}) \quad\text{and}\quad
\sqrt{H(\lambda)}\,(V_\lambda,Y_{\ell,m})_{L^2({\mathbb
  S}^{N}_+)}=O(\lambda^{\ell+1})\quad\text{as }\lambda\to0^+,
\]
for every $1\leq m\leq M_\ell$. From (\ref{2ndest}),
there exists $K>0$ such that $\sqrt{H(\lambda)}\geq
K\lambda^{\ell+\frac12}$ for $\lambda$ sufficiently small. Therefore
\begin{equation}\label{eq:26-alt}
(U_\lambda,Y_{\ell,m})_{L^2({\mathbb
  S}^{N}_+)}=
O(\lambda^{\frac32}) \quad\text{and}\quad
(V_\lambda,Y_{\ell,m})_{L^2({\mathbb
  S}^{N}_+)}=O(\lambda^{\frac12})\quad\text{as }\lambda\to0^+,
\end{equation}
for every $1\leq m\leq M_\ell$.  From Lemma \ref{l:blowup},
 for every sequence $\lambda_n\to0^+$, there exist a subsequence
$\{\lambda_{n_k}\}_{k\in\N}$ and $2M_\ell$ real constants
$\beta_{\ell,m},\beta'_{\ell,m}$, $m=1,2,\dots,M_\ell$,  such that
\begin{equation}\label{eq:35}
\sum_{m=1}^{M_\ell}((\beta_{\ell,m})^2+(\beta'_{\ell,m})^2)=1
\end{equation}
and
\[
U_{\lambda_{n_k}}\to
|z|^{\ell}
\sum_{m=1}^{M_\ell}\beta_{\ell,m}Y_{\ell,m}\Big(\frac
z{|z|}\Big),\quad
V_{\lambda_{n_k}}\to
|z|^{\ell}
\sum_{m=1}^{M_\ell}\beta'_{\ell,m}Y_{\ell,m}\Big(\frac
z{|z|}\Big), \quad\text{as }k\to+\infty,
\]
 weakly in $H^1(B_1^+)$ and hence
strongly in $L^2(S_1^+)$. It follows that, for all $m=1,2,\dots,M_\ell$,
\[
\beta_{\ell,m}=\lim_{k\to+\infty}(U_{\lambda_{n_k}},Y_{\ell,m})_{L^2({\mathbb
  S}^{N}_+)}\quad\text{and}\quad
\beta'_{\ell,m}=\lim_{k\to+\infty}(V_{\lambda_{n_k}},Y_{\ell,m})_{L^2({\mathbb
  S}^{N}_+)}
\]
and hence, in view of \eqref{eq:26-alt},
\[
\beta_{\ell,m}=0\quad\text{and}\quad \beta'_{\ell,m}=0
\quad\text{for every $m=1,2,\dots,M_\ell$},
\]
thus contradicting \eqref{eq:35}.
\end{proof}

\begin{proof}[Proof of Theorem \ref{t:asym}]
  From Lemmas \ref{l:blowup} and \ref{l:limitepositivo} there exist
  $\ell\in \N$ such that, for every sequence $\lambda_n\to0^+$, there exist a subsequence
$\{\lambda_{n_k}\}_{k\in\N}$ and $2M_\ell$ real constants
$\alpha_{\ell,m},\alpha '_{\ell,m}$, $m=1,2,\dots,M_\ell$,  such that
$\sum_{m=1}^{M_\ell}((\alpha_{\ell,m})^2+(\alpha'_{\ell,m})^2)\neq0$
and
\begin{equation}\label{eq:40}
\lambda_{n_k}^{-\ell}U(\lambda_{n_k}z)\to
|z|^{\ell}
\sum_{m=1}^{M_\ell}\alpha_{\ell,m}Y_{\ell,m}\Big(\frac
z{|z|}\Big),\quad
\lambda_{n_k}^{-\ell}V(\lambda_{n_k}z)\to
|z|^{\ell}
\sum_{m=1}^{M_\ell}\alpha'_{\ell,m}Y_{\ell,m}\Big(\frac
z{|z|}\Big),
\end{equation}
strongly in $H^1(B_r^+)$ for all $r\in(0,1)$, and then, by
homogeneity, strongly in $H^1(B_1^+)$.

From above, \eqref{eq:37}, \eqref{eq:24b}, \eqref{eq:24c},
 \eqref{eq:38ell}, \eqref{eq:36}, and \eqref{eq:39},             we deduce that
\begin{align*}
\alpha_{\ell,m}&=\lim_{k\to\infty}\lambda_{n_k}^{-\ell}
\int_{{\mathbb S}^{N}_+}U(\lambda_{n_k}\,\theta) Y_{\ell,m}(\theta)\,dS\\
&=\lim_{k\to\infty}\lambda_{n_k}^{-\ell}\varphi_{\ell,m}(\lambda_{n_k})=
c_1^{\ell,m}+\int_0^R\frac{t^{-\ell+1}}{2\ell+N-1}\widetilde
                 \varphi_{\ell,m}(t)\,dt\\
&=R^{-\ell}\int_{{\mathbb S}^{N}_+}U(R\,\theta)
  Y_{\ell,m}(\theta)\,dS-
\frac{R^{-N-2\ell+1}}{N+2\ell-1}\int_0^R t^{N+\ell}\left(\int_{{\mathbb S}^{N}_+}V(t\,\theta)
    Y_{\ell,m}(\theta)\,dS\right)\,dt\\
&\qquad\qquad
+\int_0^R\frac{t^{-\ell+1}}{2\ell+N-1}\left(\int_{{\mathbb S}^{N}_+}V(t\,\theta)
    Y_{\ell,m}(\theta)\,dS\right)\,dt
\end{align*}
and
\begin{align*}
\alpha'_{\ell,m}&=\lim_{k\to\infty}\lambda_{n_k}^{-\ell}
\int_{{\mathbb S}^{N}_+}V(\lambda_{n_k}\,\theta) Y_{\ell,m}(\theta)\,dS\\
&=\lim_{k\to\infty}\lambda_{n_k}^{-\ell}\widetilde\varphi_{\ell,m}(\lambda_{n_k})=
d_1^{\ell,m}+\int_0^R\frac{t^{-\ell+1}}{2\ell+N-1}\zeta_{\ell,m}(t)\,dt\\
&=R^{-\ell}\int_{{\mathbb S}^{N}_+}V(R\,\theta)
  Y_{\ell,m}(\theta)\,dS\\
\notag&\qquad\qquad-
\frac{R^{-N-2\ell+1}}{N+2\ell-1}\int_0^R t^{N+\ell-1}\left(
\int_{{\mathbb
      S}^{N-1}}h(t\theta')U(t\theta',0)Y_{k,m}(\theta',0)\,dS'\right)\,dt\\
&\qquad\qquad
+\int_0^R\frac{t^{-\ell}}{2\ell+N-1}\left(
\int_{{\mathbb
      S}^{N-1}}h(t\theta')U(t\theta',0)Y_{\ell,m}(\theta',0)\,dS'
\right)\,dt.
\end{align*}
We observe that the coefficients $\alpha_{\ell,m},\alpha'_{\ell,m}$ depend neither on the sequence
$\{\lambda_n\}_{n\in\N}$ nor on its subsequence
$\{\lambda_{n_k}\}_{k\in\N}$. Hence the convergences in \eqref{eq:40}  hold as $\lambda\to 0^+$
and the theorem is proved.
\end{proof}

\begin{proof}[Proof of Theorem \ref{t:sun-ext}]
Let us assume by contradiction that $(U,V)\not=(0,0)$. Then Theorem
\ref{t:asym} implies that  there exist
$\ell\in\N$ such that
\begin{equation}\label{eq:42}
\lambda^{-\ell}U(\lambda z)\to
\widehat U(\theta)
,\quad
\lambda^{-\ell}V(\lambda z)\to
\widehat  V(\theta)
,
\end{equation}
strongly in $H^1(B_1^+)$, where $(\widehat  U ,\widehat  V)\neq(0,0)$.

Assumption \eqref{eq:case2} implies that $\widehat
U\equiv0$. Hence $\widehat  V\not\equiv0$.
Let us denote $\widetilde U_\lambda(z)=\lambda^{-\ell-2}U(\lambda
z)$. Then $\widetilde U_\lambda$ satisfies
\[
-\Delta \widetilde U_\lambda(z)=\lambda^{-\ell}V(\lambda z).
\]
We have that, for all $\varphi\in C^{\infty}_{\rm c}(B_1^+)$,
\[
\lim_{\lambda\to0^+}\int_{B_1^+}\nabla \widetilde
U_\lambda(z)\cdot\nabla \varphi(z)\,dz=
\lim_{\lambda\to0^+}
\int_{B_1^+}\lambda^{-\ell}V(\lambda z)\varphi(z)\,dz
=\int_{B_1^+}\widehat  V(z)\varphi(z)\,dz.
\]
On the other, by assumption \eqref{eq:case2} we have that
\begin{align*}
\lim_{\lambda\to0^+}\int_{B_1^+}\nabla \widetilde
U_\lambda(z)\cdot\nabla \varphi(z)\,dz&=
-\lim_{\lambda\to0^+}\int_{B_1^+} \widetilde
U_\lambda(z)\Delta \varphi(z)\,dz\\
&=-\lim_{\lambda\to0^+}\lambda^{-\ell-2}\int_{B_1^+}
U(\lambda z)\Delta \varphi(z)\,dz=0.
\end{align*}
Therefore we obtain that
\[
\int_{B_1^+}\widehat  V(z)\varphi(z)\,dz=0\quad\text{for all
}\varphi\in C^{\infty}_{\rm c}(B_1^+)
\]
which implies that $\widehat  V\equiv 0$ in $B_1^+$, a contradiction.
\end{proof}

\section{Applications to fourth order problems  and higher order
  fractional  equations}\label{sec:appl-fourth-order}

In this section we discuss applications of Theorem \ref{t:asym} to
fourth order problems and higher order fractional
  equations, by proving Theorems \ref{t:asym-ho} and \ref{t:frac}.

\begin{proof}[Proof of Theorem \ref{t:asym-ho}]
From \cite[Proposition 7.2]{FF-prep}
we have that, if $U\in\mathcal D$, then $U\in H^1(B_R^+)$. Furthermore,
\cite[Proposition 2.4]{FF-prep} implies that, if $U\in\mathcal D$ is  a nontrivial weak solution
  to \eqref{eq:44} for some $h\in C^1(\Omega)$, then $V:=\Delta U$
  belongs to $H^1(B_R^+)$ for some $R>0$ so that the couple $(U,V)\in
  H^1(B_R^+)\times H^1(B_R^+)$ is
  a weak solution to
\eqref{eq:system} such that $(U,V)\not=(0,0)$. Then statement (i)
follows from Theorem \ref{t:asym} while (ii) comes from Theorem \ref{t:sun-ext}.
  \end{proof}

\begin{proof}[Proof of Theorem \ref{t:frac}]
In view of \cite{Y} (see also \cite{FF-prep}),  we have that, if $u\in
\mathcal D^{3/2,2}(\R^N)$, then there exists a unique $U\in \mathcal
D$ such that $\Delta^2 U=0$ in $\R^{N+1}_+$ and $\mathop{\rm Tr}(U)=u$
on $\R^{N+1}_+$. Moreover
\begin{equation}\label{eq:49}
\int_{\R^{N+1}_+}\Delta U(x,t) \Delta \varphi(x,t)\,dx\,dt= 2 \, (u,
\mathop{\rm Tr}\varphi)_{\mathcal D^{3/2,2}(\R^N)}
\end{equation}
 for all $\varphi\in\mathcal D$.
In particular, if $u$ solves \eqref{eq:48}, we have that $U$ is a weak solution to \eqref{eq:44}.
Let $V=\Delta U$. Since $(-\Delta)^{3/2} u\in (\mathcal D^{1/2,2}(\R^N))^*$, by \eqref{eq:49} we have that
\begin{equation} \label{eq:**}
\int_{\R^{N+1}_+} V(x,t) \Delta \varphi(x,t)\,dx\,dt=2 \phantom{a}_{(\mathcal D^{1/2,2}(\R^N))^\star}
\left\langle  (-\Delta)^{3/2} u,
\mathop{\rm Tr}\varphi\right\rangle_{\mathcal D^{1/2,2}(\R^N)}
\end{equation}
for all $\varphi\in \mathcal T$ with $\mathcal T$ as in \eqref{eq:space}. Applying \cite[Proposition 2.4]{FF-prep} to $V$ we deduce that $V\in H^1(B_r^+)$ for all $r>0$ and hence by
\eqref{eq:**} and integration by parts we obtain
\begin{equation} \label{eq:***}
-\int_{\R^{N+1}_+} \nabla V(x,t) \cdot \nabla \varphi(x,t)\,dx\,dt=2 \phantom{a}_{(\mathcal D^{1/2,2}(\R^N))^\star}
\left\langle  (-\Delta)^{3/2} u,
\mathop{\rm Tr}\varphi\right\rangle_{\mathcal D^{1/2,2}(\R^N)}
\end{equation}
for all $\varphi \in \mathcal T$.

Since the trace map $\mathop{\rm Tr}$ is
continuous from $\mathcal D^{1,2}(\R^{N+1}_+)$ into
$\mathcal D^{1/2,2}(\R^N)$, in view of assumption \eqref{eq:53} we
have that $W\mapsto
\phantom{a}_{(\mathcal D^{1/2,2}(\R^N))^\star}
\left\langle
  (-\Delta)^{3/2} u,
\mathop{\rm Tr}W\right\rangle_{\mathcal D^{1/2,2}(\R^N)}$ belongs to $(\mathcal
  D^{1,2}(\R^{N+1}_+))^\star$. Then, by
classical minimization methods, we have that the minimum
\[
\min_{W\in \mathcal
  D^{1,2}(\R^{N+1}_+)}\left[\frac12\int_{\R^{N+1}_+}|\nabla
  W(x,t)|^2\,dx\,dt+2 \phantom{a}_{(\mathcal D^{1/2,2}(\R^N))^\star}
\left\langle
  (-\Delta)^{3/2} u,
\mathop{\rm Tr}W\right\rangle_{\mathcal D^{1/2,2}(\R^N)}\right]
\]
is attained by some $\widetilde V\in  \mathcal
  D^{1,2}(\R^{N+1}_+)$ weakly solving
\begin{align}\label{eq:50}
-\int_{\R^{N+1}_+}\nabla\widetilde V(x,t) \cdot
\nabla\varphi(x,t)\,dx\,dt&=2 \,
\phantom{a}_{(\mathcal D^{1/2,2}(\R^N))^\star}
\left\langle
  (-\Delta)^{3/2} u,
\mathop{\rm Tr}\varphi\right\rangle_{\mathcal D^{1/2,2}(\R^N)}\\
\notag&=
2 \int_{\R^N}|\xi|^3 \widehat u\, \overline{\widehat{\mathop{\rm Tr}\varphi}} \,d\xi
\end{align}
for all $\varphi\in C^\infty_c(\overline{\R^{N+1}_+})$.
Combining \eqref{eq:***} and \eqref{eq:50} we infer that
\begin{equation} \label{eq:extension}
\int_{\R^{N+1}_+} \nabla(V(x,t)-\widetilde V(x,t))\cdot \nabla \varphi(x,t)\, dxdt=0 \qquad \text{for all } \varphi\in \mathcal T \, .
\end{equation}
Actually \eqref{eq:extension} still holds true for any $\varphi\in C^\infty_c(\overline{\R^{N+1}_+})$. Indeed, for any $\varphi\in C^\infty_c(\overline{\R^{N+1}})$, one can test \eqref{eq:extension} with
$\varphi_k(x,t)=\varphi(x,t)-\varphi_t(x,0)\, t\, \eta(kt)$, $k\in
\N$, where $\eta\in C^\infty_c(\R)$, $0\le \eta\le 1$, $\eta(t)=1$
for any $t\in [-1,1]$ and $\eta(t)=0$ for any $t\in
(-\infty,-2]\cup [2,+\infty)$, and pass to the limit as $k\to
+\infty$.
Therefore, if we define
\[
\widetilde W=
\begin{cases}
  V(x,t)-\widetilde V(x,t),&\text{if }t\geq0,\\
  V(x,-t)-\widetilde V(x,-t),&\text{if }t<0,
\end{cases}
\]
we easily deduce that $\int_{\R^{N+1}}\nabla \widetilde W\cdot \nabla \varphi \, dz=0$ for all $\varphi\in C^\infty_c(\R^{N+1})$. In particular $\widetilde W$ is harmonic in $\R^{N+1}$.  Furthermore, since $V\in
L^2(\R^{N+1}_+)$ and $\widetilde V\in  \mathcal
  D^{1,2}(\overline{\R^{N+1}_+})$,
we have that $\widetilde W=W_1+W_2$ for some $W_1\in L^2(\R^{N+1})$
and $W_2\in L^{\frac{2(N+1)}{N-1}}(\R^{N+1})$. The mean value property
for harmonic functions ensures that, for every $z\in  \R^{N+1}$ and $R>0$,
\begin{align*}
 |\widetilde W(z)|&= \frac1{|B(z,R)|_{N+1}}\left|\int_{B(z,R)} \widetilde
                    W(y)\,dy\right|\leq \frac{\text{\rm
                    const\,}}{R^{N+1}}
\left(\int_{B(z,R)}|W_1(y)|\,dy+\int_{B(z,R)}|W_2(y)|\,dy\right)\\
&\leq \frac{\text{\rm
                    const\,}}{R^{N+1}}\left(\|W_1\|_{L^2(\R^{N+1})}R^{\frac{N+1}{2}}+
\|W_2\|_{L^{\frac{2(N+1)}{N-1}}(\R^{N+1})}R^{\frac{N+3}{2}}\right)
\end{align*}
where $|\cdot|_{N+1}$ stands for the Lebesgue measure in $\R^{N+1}$ and
$\text{\rm const\,}$ is a positive constant independent of $z$ and $R$
which could vary from line to line. Since the right hand side of the
previous inequality tends to $0$ as $R\to+\infty$, we deduce that
$\widetilde W\equiv0$, and then $\widetilde V=V$. In particular, in
view of \cite{CS} and \eqref{eq:50}, this  implies that
\[
(v,\varphi)_{\mathcal
  D^{1/2,2}(\R^N)}=-2 \, (u,
\varphi)_{\mathcal D^{3/2,2}(\R^N)}\quad\text{for all }
\varphi\in C^\infty_c(\R^N),
\]
where we put $v=\mathop{\rm Tr}V$. This implies that $-2 \, |\xi|^3\widehat u=|\xi|\widehat v$ and hence
$v=2 \, \Delta u$ in $\R^N$.

To prove (i), it is not restrictive to assume $x_0=0$. Let us assume, by contradiction, that $u\not\equiv0$. Then
the couple $(U,V)\neq(0,0)$ is a weak solution to
\eqref{eq:system} in $H^1(B_R^+)\times H^1(B_R^+)$ for some $R>0$ with $h=-2a$.

From Theorem \ref{t:asym} it follows that
either $u$ or $v$ (which are the traces of $U$ and $V$
respectively) have vanishing order $\ell\in\N$ at $0$. In view of
assumption \eqref{eq:47} we have that necessarily $V$ vanishes of
order $\ell$, i.e. there exists $\Psi:{\mathbb S}^{N}_+\to\R$,
 a nontrivial linear combination of spherical
harmonics symmetric with respect to the equator $t=0$, such that
$\Psi\not\equiv0$ on $\partial{\mathbb S}^N_+$,
\[
\lambda^{-\ell}V(\lambda z)\to
|z|^{\ell}\Psi\Big(\frac
z{|z|}\Big) \text{ as }\lambda\to 0\text{ strongly in $H^1(B_1^+)$},
\]
and consequently
\[
\lambda^{-\ell}v(\lambda x)\to
|x|^{\ell}\Psi\Big(\frac
x{|x|},0\Big) \text{ as }\lambda\to 0\text{ strongly in $H^{1/2}(B_1')$  }.
\]
Let us denote
\[
v_\lambda(x)=\lambda^{-\ell}v(\lambda x)\quad\text{and}\quad
\widetilde u_\lambda(x)=\lambda^{-2-\ell}u(\lambda x),
\]
so that
\begin{equation}\label{eq:55}
v_\lambda\to |x|^{\ell}\Psi\Big(\frac
x{|x|},0\Big) \text{ as }\lambda\to 0\text{ strongly in
  $H^{1/2}(B_1')$}
\end{equation}
and
\[
2 \Delta \widetilde u_\lambda=v_\lambda\text{ in
}\R^N.
\]
For every $\varphi\in C^{\infty}_{\rm c}(B_1')$ we have that
\begin{equation}\label{eq:54}
-2 \int_{\R^N} \widetilde u_\lambda(-\Delta\varphi)\,dx=-2
\int_{\R^N} \varphi(-\Delta \widetilde u_\lambda)\,dx
=\int_{\R^N} \varphi v_\lambda \,dx.
\end{equation}
From one hand, assumption \eqref{eq:47} implies that
\[
\lim_{\lambda\to0^+}\int_{\R^N} \widetilde u_\lambda(-\Delta\varphi)\,dx
=0
\]
whereas convergence \eqref{eq:55} yields
\[
\lim_{\lambda\to0^+}\int_{\R^N} \varphi v_\lambda \,dx=
\int_{\R^N} |x|^{\ell}\Psi\Big(\frac
x{|x|},0\Big) \varphi(x)\,dx.
\]
Hence passing to the limit in \eqref{eq:54} we obtain that
\[
\int_{\R^N} |x|^{\ell}\Psi\Big(\frac
x{|x|},0\Big) \varphi(x)\,dx=0\quad\text{for every }\varphi\in
C^{\infty}_{\rm c}(B_1'),
\]
thus contradicting  the fact that $|x|^{\ell}\Psi\Big(\frac
x{|x|},0\Big) \not\equiv0$.

To prove (ii), let us assume by contradiction, that $u\not\equiv0$ in
$\Omega$  and
$u(x)=0$ a.e. in a set $E\subset \Omega$ with $|E|_N>0$, where $|\cdot|_N$ denotes the
$N$-dimensional Lebesgue measure.
Since $2 \Delta u=v$ and $v\in \mathcal  D^{1/2,2}(\R^N)\subset
L^{2}_{\rm loc}(\R^N)$, by classical regularity theory we have that
$u\in H^2_{\rm loc}(\Omega)$. Since $u(x)=0$ for a.e. $x\in E$, we
have that $\nabla u(x)=0$ for a.e. $x\in E$ and hence, since
$\frac{\partial u}{\partial x_i}\in  H^1_{\rm loc}(\Omega)$ for every
$i$, $\Delta u=0$ a.e. in $E$. In particular there exists  a set
$E'\subset E\subset\Omega$
 with $|E'|_N>0$ such that $u(x)=\Delta u(x)=0$ a.e. in
 $E'$. In particular
$v(x)=0$ a.e. in
 $E'$.

By Lebesgue's density Theorem,
a.e. point of $E'$ is a density point of $E'$. Let  $x_0$ be a
density point of $E'$. Hence,  for all $\e>0$ there exists $r_0=r_0(\e)\in(0,1)$
such that, for all $r\in(0,r_0)$,
\begin{equation}\label{eq:56}
\frac{|(\R^N\setminus E')\cap B'_r(x_0)|_N}{|B'_r(x_0)|_N}<\e,
\end{equation}
where $B'_r(x_0)=\{x\in\R^N:|x-x_0|<r\}$.
Theorem \ref{t:asym} implies that
 there exist $\Psi_1,\Psi_2:{\mathbb S}^{N}_+\to\R$ linear combination of spherical
harmonics such that  either $\Psi_1\not\equiv0$ or $\Psi_2\not\equiv0$
and
\begin{equation}\label{eq:57}
\lambda^{-\ell}u(x_0+\lambda (x-x_0) )\to
|x-x_0|^{\ell}\Psi_1\Big(\frac{x-x_0}{|x-x_0|},0\Big)
\end{equation}
and
\begin{equation}\label{eq:58}
\lambda^{-\ell}v(x_0+\lambda (x-x_0) )\to
|x-x_0|^{\ell}\Psi_2\Big(\frac{x-x_0}{|x-x_0|},0\Big)
\end{equation}
as $\lambda\to 0$ strongly in $H^{1/2}(B_1'(x_0))$ and then,  by the Sobolev embedding
$H^{\frac12}(B_1'(x_0))\hookrightarrow L^{\frac{2N}{N-1}}(B_1'(x_0))$,  strongly in $L^{\frac{2N}{N-1}}(B_1'(x_0))$

Since $u\equiv v\equiv 0$ in $E'$,
 by \eqref{eq:56} we have
\begin{align*}
\int_{B'_r(x_0)}u^2(x)\,dx&=\int_{(\R^N\setminus E')\cap
  B'_r(x_0)}u^2(x)\,dx\\
&\leq \bigg(\int_{(\R^N\setminus E')\cap
  B'_r(x_0)}|u(x)|^{2N/(N-1)}dx\bigg)^{\!\!\frac{N-1}{N}}|(\R^N\setminus E') \cap
B'_r(x_0)|_N^{1/N}\\
&<\e^{1/N}|B'_r(x_0)|_N^{1/N}
\bigg(\int_{(\R^N\setminus E')\cap
  B'_r(x_0)}|u(x)|^{2N/(N-1)}dx\bigg)^{\!\!\frac{N-1}{N}}
\end{align*}
and similarly
\begin{align*}
\int_{B'_r(x_0)}v^2(x)\,dx<\e^{1/N}|B'_r(x_0)|_N^{1/N}
\bigg(\int_{(\R^N\setminus E')\cap
  B'_r(x_0)}|v(x)|^{2N/(N-1)}dx\bigg)^{\!\!\frac{N-1}{N}}
\end{align*}
for all $r\in(0,r_0)$. Then, letting
$u^r(x):=r^{-\ell}u(x_0+r(x-x_0))$ and $v^r(x):=r^{-\ell}v(x_0+r(x-x_0))$,
 \begin{align*}
&\int_{B'_1(x_0)}|u^r(x)|^2dx<\Big(\frac{\omega_{N-1}}N\Big)^{\!\frac1N}
\e^{\frac1N}\bigg(\int_{B'_1(x_0)}|u^r(x)|^{\frac{2N}{N-1}}dx\bigg)^{\!\!\frac{N-1}{N}},\\
&\int_{B'_1(x_0)}|v^r(x)|^2dx<\Big(\frac{\omega_{N-1}}N\Big)^{\!\frac1N}
\e^{\frac1N}\bigg(\int_{B'_1(x_0)}|v^r(x)|^{\frac{2N}{N-1}}dx\bigg)^{\!\!\frac{N-1}{N}},
\end{align*}
for all $r\in(0,r_0)$, where $\omega_{N-1}=\int_{{\mathbb S}^{N-1}}1\,dS'$. Letting $r\to 0^+$, from \eqref{eq:57} and \eqref{eq:58}  we have that
 \begin{multline*}
\int_{B'_1(x_0)}
|x-x_0|^{2\ell}\Psi_i^2\Big(\tfrac{x-x_0}{|x-x_0|},0\Big)\,dx\\
\leq
\Big(\frac{\omega_{N-1}}N\Big)^{\!\frac1N}
\e^{\frac1N}\bigg(\int_{B'_1(x_0)}|x-x_0|^{\frac{2N \ell }{N-1}}\left|
\Psi_i\Big(\tfrac{x-x_0}{|x-x_0|},0\Big)\right|^{\frac{2N}{N-1}}dx\bigg)^{\!\!\frac{N-1}{N}}
\quad\text{for }i=1,2,
\end{multline*}
which yields a contradiction as $\e\to 0^+$, since
either $\Psi_1\not\equiv0$ or $\Psi_2\not\equiv0$.
\end{proof}

\end{document}